\pdfoutput=1

\documentclass[preprint,10pt]{elsarticle}
\usepackage{fullpage}

\usepackage{amsmath,amssymb,amsfonts,mathrsfs,amsthm}
\usepackage[titletoc,toc,title]{appendix}

\usepackage{array}
\usepackage{listings}
\usepackage{mathtools}
\usepackage{pdfpages}
\usepackage[textsize=footnotesize,color=green]{todonotes}
\usepackage{bm}
\usepackage[normalem]{ulem}
\usepackage{hhline}

\usepackage{algorithm}
\usepackage[noend]{algpseudocode}
\usepackage{algorithmicx}
\algblock{ParFor}{EndParFor}
\algnewcommand\algorithmicparfor{\textbf{parfor}}
\algnewcommand\algorithmicpardo{\textbf{do}}
\algnewcommand\algorithmicendparfor{\textbf{end\ parfor}}
\algrenewtext{ParFor}[1]{\algorithmicparfor\ #1\ \algorithmicpardo}
\algrenewtext{EndParFor}{\algorithmicendparfor}
\usepackage{graphicx}
\usepackage{subfig}
\usepackage{color}

\usepackage{pgfplots}
\usepackage{pgfplotstable}
\definecolor{markercolor}{RGB}{124.9, 255, 160.65}
\pgfplotsset{width=10cm,compat=1.3}
\pgfplotsset{
tick label style={font=\small},
label style={font=\small},
legend style={font=\small}
}

\newcommand{\vect}[1]{\ensuremath\boldsymbol{#1}}

\newcommand{\td}[2]{\frac{{\rm d}#1}{{\rm d}{\rm #2}}}
\newcommand{\pd}[2]{\frac{\partial#1}{\partial#2}}

\newcommand{\mbb}[1]{\mathbb{#1}}

\newcommand{\LRp}[1]{\left( #1 \right)}
\newcommand{\LRs}[1]{\left[ #1 \right]}

\newcommand{\LRc}[1]{\left\{ #1 \right\}}

\newcommand{\Grad} {\ensuremath{\nabla}}
\newcommand{\Div} {\ensuremath{\nabla\cdot}}

\newcommand{\jump}[1] {\ensuremath{\LRs{\![#1]\!}}}
\newcommand{\avg}[1] {\ensuremath{\LRc{\!\{#1\}\!}}}

\newcommand{\Gh}{\Gamma_h}

\newcommand{\Oh}{\Omega_h}

\newtheorem{theorem}{Theorem}[section]
\newtheorem{lemma}[theorem]{Lemma}

\newenvironment{definition}[1][Definition]{\begin{trivlist}
\item[\hskip \labelsep {\bfseries #1}]}{\end{trivlist}}

\newcommand{\eval}[2][\right]{\relax
  \ifx#1\right\relax \left.\fi#2#1\rvert}

\newcommand{\edit}[1]{#1}

\newcolumntype{C}[1]{>{\centering\let\newline\\\arraybackslash\hspace{0pt}}m{#1}}

\newcommand*\diff[1]{\mathop{}\!{\mathrm{d}#1}}

\makeatletter
\renewcommand\d[1]{\mspace{6mu}\mathrm{d}#1\@ifnextchar\d{\mspace{-3mu}}{}}
\makeatother

\begin{document}

%\maketitle
\begin{frontmatter}
\author[rice]{Jesse Chan\corref{cor1}}
\ead{Jesse.Chan@rice.edu}
\cortext[cor1]{Principal Corresponding author}
\author[rice]{Zheng Wang}
\ead{zw14@caam.rice.edu}
\author[total]{Russell J.\ Hewett}
\ead{russell.hewett@total.com}
\author[vt]{T. Warburton}
\ead{tcew@vt.edu}
\address[rice]{Department of Computational and Applied Mathematics, Rice University, 6100 Main St, Houston, TX, 77005}
\address[vt]{Department of Mathematics, Virginia Tech, 225 Stanger Street Blacksburg, VA, 24061-0123}
\address[total]{TOTAL E\&P Research and Technology USA, 1201 Louisiana St, Houston, TX 77002}

\title{Reduced storage nodal discontinuous Galerkin methods on semi-structured prismatic meshes}

\begin{abstract}
We present a high order time-domain nodal discontinuous Galerkin method for wave problems on hybrid meshes consisting of both wedge and tetrahedral elements.  We allow for vertically mapped wedges which can be deformed along the extruded coordinate, and present a simple method for producing quasi-uniform wedge meshes for layered domains.  We show that standard mass lumping techniques result in a loss of energy stability on meshes of vertically mapped wedges, and propose an alternative which is both energy stable and efficient.  High order convergence is demonstrated, and comparisons are made with existing low-storage methods on wedges.  Finally, the computational performance of the method on Graphics Processing Units is evaluated.  
\end{abstract}

\end{frontmatter}

%\begin{frontmatter}
%\journal{Some ELS Journal}
%\title{Nodal discontinuous Galerkin methods on vertically mapped wedges}
%\author{Jesse Chan}
%\ead{jlchan@vt.edu}
%\author{Russell Hewett}
%\author{Zheng Wang}
%\author{T. Warburton}

%\begin{abstract}
%We present a high order time-domain nodal discontinuous Galerkin method for wave problems on hybrid meshes consisting of both wedge and tetrahedral elements.  We allow for vertically mapped wedges which can be deformed along the extruded coordinate, and present a simple method for producing quasi-uniform wedge meshes for layered domains.  We show that standard mass lumping techniques result in a loss of energy stability on meshes of vertically mapped wedges, and propose an alternative which is both energy stable and efficient.  High order convergence is demonstrated, and comparisons are made with existing low-storage methods on wedges.  Finally, the computational performance of the method on Graphics Processing Units is evaluated.  
%\end{abstract}
%\end{frontmatter}

%\tableofcontents

\section{Introduction}

While time-dependent wave propagation and seismic applications have traditionally utilized high order finite difference solvers, interest in high order finite element methods in seismic imaging has grown over the past two decades, in part due to the ease of adapting finite element methods to complex geometries.  However, compared to finite difference methods, a drawback of traditional finite element methods for time-domain wave propagation problems is that a global mass matrix system must be solved at every timestep.  Several new high order finite element methods were introduced to address this issue.  Among these methods, a commonly used high order method is the spectral element method (SEM) \cite{komatitsch1999introduction, karniadakis2013spectral}, which avoids the solution of a global system by employing a diagonal mass-lumped approximation to the mass matrix.  However, the use of SEM is currently limited to unstructured hexahedral meshes, which can be difficult to construct for arbitrary topologies.  

Unlike spectral elements methods, high order discontinuous Galerkin (DG) methods \cite{hesthaven2007nodal} can use more general meshes.  Like spectral methods, DG methods avoid the inversion of a globally coupled mass matrix.  Additionally, due to their low communication costs and high arithmetic intensity, time-domain DG methods have also been shown to be well-suited to parallelization, both on distributed and many-core architectures.  Kl\"ockner et al.\ presented an efficient implementation of quadrature-free nodal DG methods for tetrahedral meshes on a single Graphics Processing Unit (GPU) \cite{klockner2009nodal}, while the scalability of the multi-GPU case was investigated in \cite{godel2010scalability}.  

The utility and efficiency of GPU-accelerated nodal discontinuous Galerkin methods for seismic applications on tetrahedral meshes is explored in detail by Modave et al.\ in \cite{modave2015nodal, modave2016gpu}.  However, the GPU implementation of DG methods on non-tetrahedral elements requires some care, as storage requirements for such elements can be quite large, especially at high orders of approximation.  Since many-core accelerators typically have a relatively small amount of on-device storage, storage-heavy schemes can greatly limit the maximum problem size which is executable on a single GPU.  Furthermore, this can be detrimental to performance, as accelerators such as GPUs typically require sufficiently large problem sizes to operate at peak efficiency.  

In \cite{chan2015gpu}, a GPU-accelerated time-domain DG method was introduced for hybrid meshes consisting of hexahedra, tetrahedra, wedge (or prismatic) and pyramid elements.  This method addressed storage costs on non-tetrahedral elements.  However, to accomodate general wedge elements, the method presented in \cite{chan2015gpu} requires the use the rational Low-storage Curvilinear DG (LSC-DG) basis \cite{warburton2013low}, which mitigates the need to explicitly store elemental matrices for wedges in DG methods.  This choice of basis for wedge elements in turn requires the use of quadrature-based \textit{a priori} stable variational formulations to maintain energy stability.  Additionally, high order convergence under LSC-DG bases requires the use of mesh refinement strategies which result in asymptotically affine elements.%, which may not hold true for non-nested (unstructured) sequences of meshes.  

In this paper, we investigate the efficient GPU implementation of nodal DG methods on meshes consisting of both wedge and tetrahedral elements for geophysical applications.  The focus is in the efficient solution of wave problems on domains which exhibit either layered features or anisotropy in one direction.  Domains of this nature are well-suited to meshes of wedge elements.  For example, Lapin et al.\ \cite{lapin2003joint} utilized meshes consisting of wedges for seismic inversion; however, their focus was in the resolution of interfaces in velocity models, while the focus of this work is the efficient implementation of a time-domain solver for wave propagation and geophysical applications on such meshes.  Prismatic meshes have also been adopted in marine applications for resolving topographical variations at ocean floors \cite{wang2008finite, blaise2010discontinuous}. 

In this work, we restrict ourselves to semi-structured prismatic meshes consisting of \textit{vertically mapped} wedges, which are deformed from their reference configuration only along the vertical (extruded) coordinate.  By taking advantage of the geometric mappings of such elements, we construct a limited storage, quadrature-free nodal DG method on meshes of tetrahedra and vertically mapped wedge elements.  The proposed method addresses issues of robustness to unstructured mesh refinement present for the LSC-DG wedge basis \cite{warburton2013low, chan2015gpu}, and issues of energy stability associated with mass-lumped approximations.  

The paper is structured as follows: we review the continuous and discrete formulations of time domain DG methods for the  wave equation in Section~\ref{sec:form}.  In Section~\ref{sec:space}, we show properties of {vertically mapped} wedge elements which allow for simplifications of the elemental matrices associated with time domain DG methods, and present a limited-storage method for the acoustic wave equation on such meshes.  In Section~\ref{sec:mesh}, we present a simple method to construct hybrid meshes of vertically mapped wedges and tetrahedra for layered media.  Section~\ref{sec:num} gives numerical results confirming the high order convergence of our method, as well as an analysis of the computational performance of the solver at different orders of approximation.  Comparisons are drawn between the efficiency of the proposed limited-storage method and LSC-DG on semi-structured prismatic meshes.

\section{Discontinuous Galerkin formulation}
\label{sec:form}
This work is concerned with numerical solutions of the acoustic wave equation.  This equation is posed on a domain $\Omega$  with boundary $\partial \Omega$.  In first order form, the homogeneous acoustic wave equation is
\begin{align*}
\frac{1}{\kappa}\pd{p}{\tau}{} + \Div \bm{u} &= 0\\
\rho\pd{\bm{u}}{\tau}{} + \Grad p &= 0,
\end{align*}
where $\tau$ is time, $p$ is pressure, $\bm{u}$ is the vector velocity $(u_x,u_y,u_z)$, and $\rho$ and $\kappa$ are density and bulk modulus, respectively.  We assume also a triangulation of the domain $\Omega_h$ consisting of elements $D^k$ with faces $f$.  Additionally, we assume that $\rho$ and $\kappa$ are constant over each element, though they may vary spatially from element to element.  

Let $f$ be a face between \edit{an element $D^k$ and its neighbor $D^{k^+}$}. We define $(p^-,\bm{u}^-)$ as the restriction of the pressure and velocity respectively on the element $D^k$ to the face $f$.  \edit{Similarly, we define test functions $(\phi^-, \bm{\psi}^-)$ with support on $D^k$}.  Likewise, we define $(p^+,\bm{u}^+)$ as the restriction of the solution on the element $D^{k^+}$ to $f$.  The jump of $p$, the normal jump of $\bm{u}$, and average of $\bm{u}$ are then defined as
\[
\jump{p} = p^+ - p^-, \qquad \jump{\bm{u}} = \bm{u}^+ - \bm{u}^-, \qquad \avg{\bm{u}} = \frac{\bm{u}^+ + \bm{u}^-}{2}.
\]
By defining $\vect{n}^-$ as the outward unit normal on a given face, the variational formulation for the discontinuous Galerkin method is given over each element $D^k$ as
\begin{align}
\int_{D^k} \frac{1}{\kappa}\pd{p}{\tau}{}\phi^- \diff x &= -\int_{D^k} \edit{\LRp{\Div\bm{u}}}\phi^- \diff x + \int_{\partial D^k} \frac{1}{2}\LRp{\tau_p\jump{p} - \bm{n}\cdot \jump{\bm{u}} }\phi^- \diff x  \nonumber \\ 
\int_{D^k} \rho\pd{\bm{u}}{\tau}{}\bm{\psi}^- \diff x &= - \int_{D^k} \Grad p \cdot  \bm{\psi}^- \diff x + \int_{\partial D^k} \frac{1}{2}\LRp{\tau_u \jump{\bm{u}}\cdot \bm{n}^- - \jump{p}}\bm{\psi}^-\cdot \bm{n}^- \diff x,
\label{eq:form}
\end{align}
 where $c^2 = \kappa/\rho$ is the speed of sound and $\tau_p = 1/\avg{\rho c}$, $\tau_u = \avg{\rho c}$.  This formulation can be shown to be energy stable for any choice of $\tau_p, \tau_u$ \cite{warburton2013low}, and the specific choice of $\tau_p, \tau_u$ reduces the numerical flux to the upwind fluxes (as determined by the solution of a Riemann problem) for constant $\rho, c$.  
% and the numerical fluxes in the above formulation are determined through the exact solution of a Riemann problem \cite{modave2015nodal}.  % \edit{check the units and preciseness of the flux}.  
The above formulation is referred to as the ``strong'' DG formulation \cite{hesthaven2007nodal}, where integration by parts is used to ensure all derivatives are applied to the solution variables instead of the weighting functions $\phi$ and $\bm{\psi}$.  

Finally, for this work, we assume homogeneous Dirichlet boundary conditions $p=0$ on $\partial \Omega$.  These are enforced through reflection conditions at boundary faces $f \in \partial \Omega$
\[
\left.p^+\right|_{f} = -\left.p^-\right|_{f}, \qquad \left.\bm{u}^+\right|_{f} = \left.\bm{u}^-\right|_{f}.
\]

\subsection{Discrete formulation}
\label{sec:discrete}

We seek polynomial solutions $(p,\bm{u})$ over each element $D^k$ such that $(p,\bm{u})$ satisfy (\ref{eq:form}).  We assume that $p\in V_h\LRp{D^k}$ and ${u}_x, u_y,u_z \in V_h\LRp{D^k}$ , where $V_h\LRp{D^k}$ is a local approximation space on $D^k$ with dimension $N_p$ and basis $\LRc{\phi_i}_{i=1}^{N_p}$.  In Section~\ref{sec:space}, we make precise the space spanned by this basis, and in Section~\ref{sec:basis} we specify to nodal (Lagrange) polynomial bases.   

We assume that each physical element $D^k$ is the image of a reference element $\widehat{D}$, and that a geometric mapping $\bm{\Phi}^k$ exists such that 
\[
\LRp{x,y,z} = \bm{\Phi}^k\LRp{r,s,t}, \qquad \LRp{x,y,z}\in D^k, \qquad \LRp{r,s,t} \in \widehat{D}. 
\]  
We assume non-degenerate mappings such that the Jacobian $J^k$ of this geometric mapping is greater than zero.  Physical approximation spaces are then defined over each element as mappings of a reference approximation space as follows.  Let $\LRc{\phi_i}_{i=1}^{N_p}$ be some basis for a reference approximation space over $\widehat{D}$; then, the physical approximation space over $D^k$ is defined by mappings of this reference space
\[
V_h\LRp{D^k} = {\rm span}\LRc{ \phi_1 \circ \LRp{\bm{\Phi}^k}^{-1}, \ldots, \phi_{N_p} \circ \LRp{\bm{\Phi}^k}^{-1}},
\]
\edit{and each physical basis function is defined as the composition of a reference basis function and the mapping $\bm{\Phi}^k$ to $D^k$.}
Derivatives are computed through reference derivatives and the chain rule.  %; assuming $p \in V^h_k$, $\pd{p}{x}$ on $D^k$ is computed as 
%\[
%\pd{p}{x} = \pd{r}{x} \pd{p}{r} + \pd{r}{x} \pd{p}{s} + \pd{t}{x} \pd{p}{t}.
%\]
%Derivatives in $y, z$ coordinates are computed similarly.  

The discrete formulation of the DG method is given most simply in terms of mass, (weak) derivative, and lift matrices.  The mass matrix $\bm{M}^k$ and face mass matrix $\bm{M}^{k}_f$ for the element $D^k$ are defined as
\begin{align*}
\LRp{\bm{M}^k}_{ij} &= \int_{D^k}\phi_j \phi_i = \int_{\widehat{D}}{ \phi_j \phi_i} J^k\\
\LRp{\bm{M}^{k}_f}_{ij} &= \int_{\partial D^{k, f}} \phi_j \phi_i = \int_{\widehat{D}_f}  \phi_j \phi_i J^{k}_f.
\end{align*}
where $J^{k}_f$ is the Jacobian of the mapping from the face of a reference element $\widehat{D}^k$ to the face of a physical element $D^{k}_f$. 

We also define weak differentiation matrices $\bm{S}^k_x, \bm{S}^k_y, \bm{S}^k_z$ with entries
\begin{align*}
\LRp{\bm{S}^k_x}_{ij} = \int_{\widehat{D}} \pd{\phi_j}{x} \phi_i J^k, \qquad \LRp{\bm{S}^k_y}_{ij} = \int_{\widehat{D}} \pd{\phi_j}{y} \phi_i J^k, \qquad \LRp{\bm{S}^k_z}_{ij} = \int_{\widehat{D}} \pd{\phi_j}{z} \phi_i J^k.
\end{align*}

The discrete DG formulation \edit{over the $k$th element} may then be given in terms of these matrices
\begin{align*}
\frac{1}{\kappa}\edit{\bm{M}^k}\td{\bm{p}}{\tau} &= \LRp{\bm{S}_x^k\bm{U}_x + \bm{S}_y^k\bm{U}_y + \bm{S}_z^k\bm{U}_z}  + \sum_{f=1}^{N_{\text{faces}}}\bm{M}^k_f F_p(\bm{p}^-,\bm{p}^+,\bm{U}^-,\bm{U}^+),\\
\rho \edit{\bm{M}^k}\td{\bm{U}_x}{\tau} &= \bm{S}_x^k \bm{p} + \sum_{f=1}^{N_{\text{faces}}}  \bm{n}_x\bm{M}^k_f  F_{u}(\bm{p}^-,\bm{p}^+,\bm{U}^-,\bm{U}^+)\\
\rho \edit{\bm{M}^k}\td{\bm{U}_y}{\tau} &= \bm{S}_y^k \bm{p} + \sum_{f=1}^{N_{\text{faces}}}  \bm{n}_y\bm{M}^k_f  F_{u}(\bm{p}^-,\bm{p}^+,\bm{U}^-,\bm{U}^+)\\
\rho \edit{\bm{M}^k}\td{\bm{U}_z}{\tau} &= \bm{S}_z^k \bm{p} + \sum_{f=1}^{N_{\text{faces}}}  \bm{n}_z\bm{M}^k_f  F_{u}(\bm{p}^-,\bm{p}^+,\bm{U}^-,\bm{U}^+).
\end{align*}
where $\bm{U}_x,\bm{U}_y,\bm{U}_z$ are nodal values for $u_x,u_y,u_z$.  Let $D^{k^+}$ be the neighbor to $D^k$ across a face $f$; then, $\bm{p}^-$ \edit{are} face nodal values for $p$ on $D^k$, while $\bm{p}^+$ are face nodal values for $p$ on $D^{k^+}$.    The fluxes $F_p,F_u$ are defined such that %$F^1(\bm{p}^-,\bm{p}^+)$ and $F^2(\bm{p}^-,\bm{p}^+)$ are
\begin{align*}
\LRp{ \bm{M}^k_f F_p(\bm{p}^-,\bm{p}^+,\bm{U}^-,\bm{U}^+)}_i &= \int_{f_{D^k}} \frac{1}{2}\LRp{\tau_p \jump{p} - \bm{n}^-\cdot\jump{\bm{u}}}\phi_i^-\\ 
\LRp{ \bm{n}_i \bm{M}^k_f F_u(\bm{p}^-,\bm{p}^+,\bm{U}^-,\bm{U}^+)}_i &= \int_{f_{D^k}} \frac{1}{2}\LRp{\tau_u\jump{\bm{u}} \cdot \bm{n}^- - \jump{p}}\bm{\psi}_i^- \bm{n}_i^-.
%\LRp{ \bm{M}^k_f F_p(\bm{p}^-,\bm{p}^+,\bm{U}^-,\bm{U}^+)}_i &= \int_{\partial D^k} \frac{1}{\rho^+c^+ + \rho^-c^-}\LRp{\jump{p} - \rho^+c^+\bm{n}\cdot \jump{\bm{u}} }\phi^- \diff x  \nonumber \\ 
%\LRp{ \bm{n}_i \bm{M}^k_f F_u(\bm{p}^-,\bm{p}^+,\bm{U}^-,\bm{U}^+)}_i &= \int_{\partial D^k} \frac{\rho^+c^+ \rho^-c^-}{\rho^+c^+ + \rho^-c^-}\LRp{ \jump{\bm{u}}\cdot \bm{n}^- - \frac{1}{\rho^+c^+}\jump{p}}\bm{\psi}^-\cdot \bm{n}^- \diff x.
\end{align*}
Inverting $\bm{M}^k$ results in the system of ODEs
\begin{align*}
%\label{eq:discrete_var}
\frac{1}{\kappa}\td{\bm{p}}{\tau} &= \LRp{\bm{D}_x^k + \bm{D}_y^k + \bm{D}_z^k} \bm{U}_j + \sum_{f=1}^{N_{\text{faces}}}\bm{L}^k_f F_p(\bm{p}^-,\bm{p}^+,\bm{U}^-,\bm{U}^+),\\
\rho \td{\bm{U}_x}{\tau} &= \bm{D}_x^k \bm{p} + \sum_{f=1}^{N_{\text{faces}}}  \bm{n}_x\bm{L}^k_f  F_{u}(\bm{p}^-,\bm{p}^+,\bm{U}^-,\bm{U}^+)\\
\rho \td{\bm{U}_y}{\tau} &= \bm{D}_y^k \bm{p} + \sum_{f=1}^{N_{\text{faces}}}  \bm{n}_y\bm{L}^k_f  F_{u}(\bm{p}^-,\bm{p}^+,\bm{U}^-,\bm{U}^+)\\
\rho \td{\bm{U}_z}{\tau} &= \bm{D}_z^k \bm{p} + \sum_{f=1}^{N_{\text{faces}}}  \bm{n}_z\bm{L}^k_f  F_{u}(\bm{p}^-,\bm{p}^+,\bm{U}^-,\bm{U}^+).
\end{align*}
where we have introduced differentiation matrices $\bm{D}^k_x, \bm{D}^k_y, \bm{D}^k_z$  on the element $D^k$
\[
\bm{D}^k_x = \LRp{\bm{M}^k}^{-1} \bm{S}^k_x, \qquad \bm{D}^k_y = \LRp{\bm{M}^k}^{-1} \bm{S}^k_y, \qquad \bm{D}^k_z = \LRp{\bm{M}^k}^{-1} \bm{S}^k_z
\]
and lift matrices $\bm{L}^k_f$ for each face $f$ of the element $D^k$
\[
\bm{L}^k_f = \LRp{\bm{M}^k}^{-1}\bm{M}^k_f.%, \qquad \bm{L}^{f}_i = \bm{M}^{-1}(\bm{D}^i)^T\bm{M}^k_f.
\]
Using explicit time stepping methods, the solution can be evolved in time through evaluations of local right hand sides for each element, without needing to invert a large global system of equations.  

\section{Hybrid tet-wedge meshes}
\label{sec:space}
In this section, we specialize to meshes whose elements consist of tetrahedra and a restricted class of wedges.  %Additionally, we restrict ourselves to wedge elements which are mapped only in the vertical $z$-direction.  
We define the space of degree $N$ polynomials over a tetrahedron $\mathcal{T}$ as 
\[
P^N\LRp{\mathcal{T}} = \LRc{x^i y^j z^k, \qquad 0\leq i+j+k\leq N}, \qquad \LRp{x,y,z} \in \mathcal{T}.  
\]
This space has dimension $N_p = (N+1)(N+2)(N+3)/6$.  The space of degree $N$ polynomials for a wedge $\mathcal{W}$ is defined similarly
\[
P^N\LRp{\mathcal{W}} = \LRc{x^iy^jz^k, \qquad 0\leq i+j\leq N, \qquad 0 \leq k \leq N}, \qquad \LRp{x,y,z} \in \mathcal{W},
\]
with dimension $N_p = (N+1)^2(N+2)/2$.  Alternatively, the space of polynomials for the wedge may be defined as the tensor product of degree $N$ polynomials over the reference bi-unit triangle $\widehat{T}$
\[
P^N\LRp{\widehat{T}} = \LRc{x^i y^j, \quad 0\leq i+j\leq N}, \qquad \LRp{x,y} \in \widehat{T} =  \LRc{ -1 \leq y \leq 1, -1 \leq x \leq 1-y}
\]
and the space of degree $N$ polynomials $P^N\LRp{[-1,1]}$ over the bi-unit interval.  

In this work, we assume low-order geometric mappings for all elements.  We refer to such elements as \emph{vertex-mapped}, since specifying physical vertex positions uniquely determines $\bm{\Phi}^k$ for $\bm{\Phi}^k \in P^1$ for tetrahedral elements and $\bm{\Phi}^k \in P^1(\widehat{T}) \times P^1([-1,1])$ for wedge elements.    Additionally, we assume high order reference approximation spaces for both wedges and tetrahedra.  Under this assumption, the physical approximation space $V_h\LRp{D^k}$ over $D^k$ is 
\[
V_h\LRp{D^k} = \bm{\Phi}^k \circ V_h\LRp{\widehat{D}}.
\]
which guarantees high order accuracy through reproducability of polynomials of total degree $N$ \cite{botti2012influence}.  

\subsection{Nodal bases for tetrahedra and wedges}
\label{sec:basis}
In this work, we take the basis $\LRc{\phi_i}_{i=1}^{N_p}$ to be a nodal (Lagrange) polynomial basis $\LRc{\ell_i}_{i=1}^{N_p}$, such that for a set of points $\LRc{r_i,s_i,t_i}_{i=1}^{N_p} \in \widehat{D}$, 
\[
\ell_j(r_i,s_i,t_i) = \delta_{ij}, \qquad 1\leq i,j \leq N_p.
\]
For quadrature-free DG methods, nodal bases reduce the work involved in the computation of numerical fluxes at face points, since nodal values on a given face can be extracted from the degrees of freedom without additional computation.  

For the tetrahedron, nodal bases are defined using Warp and Blend points \cite{warburton2006explicit}, which are optimized for interpolation quality and numerical stability.  For wedges, a nodal basis is defined by taking the tensor product of nodes on the triangular face of the tetrahedron with degree $N$ Gauss-Legendre-Lobatto quadrature nodes in the vertical direction, as shown in Figure~\ref{fig:wn}.  Since wedges and tetrahedra are connected through triangular faces, this ensures conformity between wedge and tetrahedral elements.  This construction also implies that Lagrange wedge basis functions are tensor products of triangular and one-dimensional Lagrange basis functions $\ell^{\rm tri}_i(r,s)$ and $\ell^{\rm 1D}(t)$ (respectively)
\[
\ell_{ij}(r,s,t) = \ell^{\rm tri}_i(r,s) \ell^{\rm 1D}_j(t).
\]
We index nodal basis functions on the wedge by $ij$
\[
\ell_{ij}(r,s,t), \qquad 0\leq i \leq (N+1)(N+2)/2, \qquad 0\leq j \leq N,
\]
where $(N+1)(N+2)/2$ is the number of nodes in a triangular ``slice'' of the wedge, $i$ is the index for nodes within a triangular slice, and $j$ enumerates the nodes along the extruded coordinate of the wedge.  

\begin{figure}
\centering
\subfloat{\includegraphics[width=.425\textwidth]{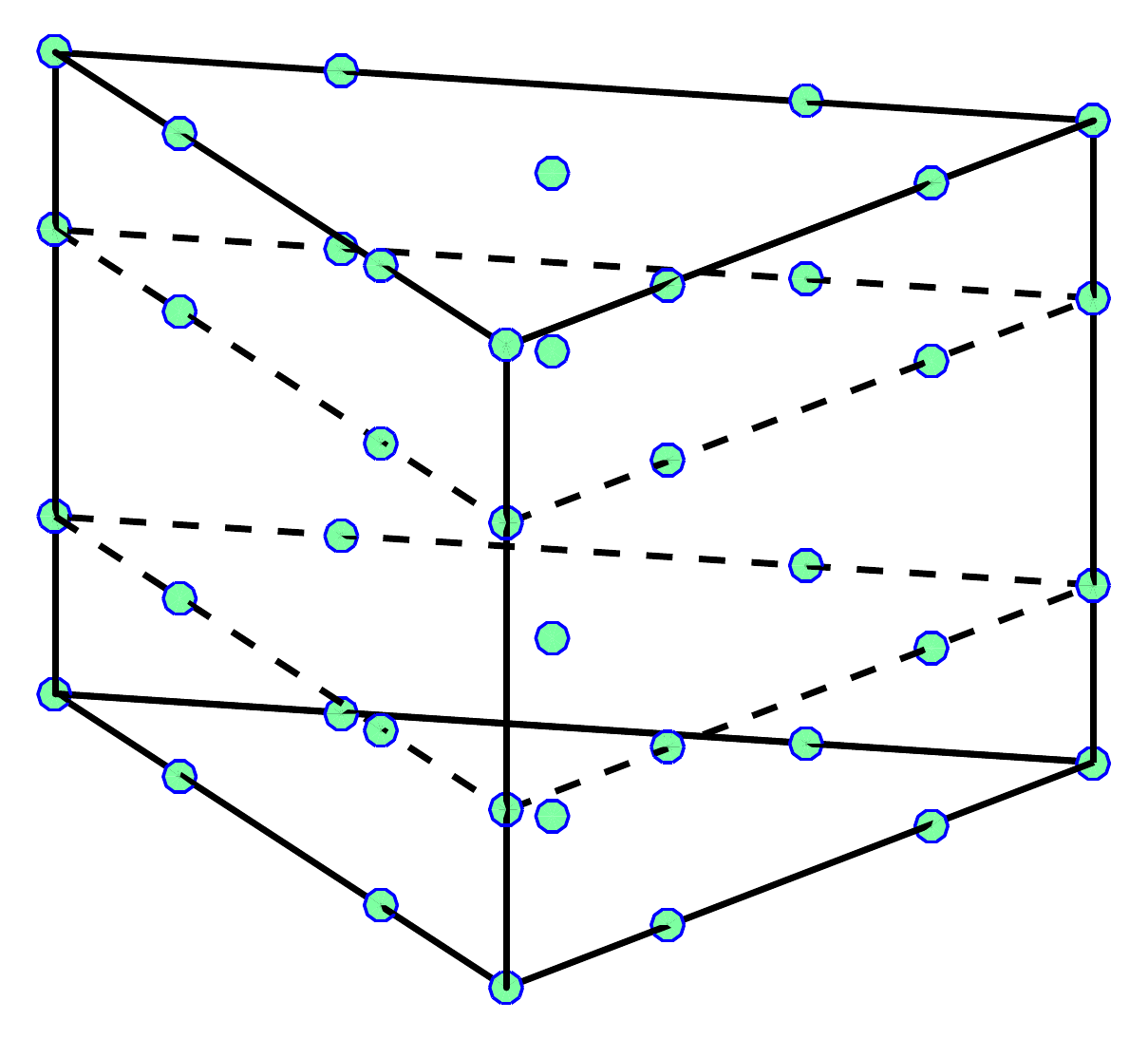}}
\hspace{4em}
\subfloat{\includegraphics[width=.3\textwidth]{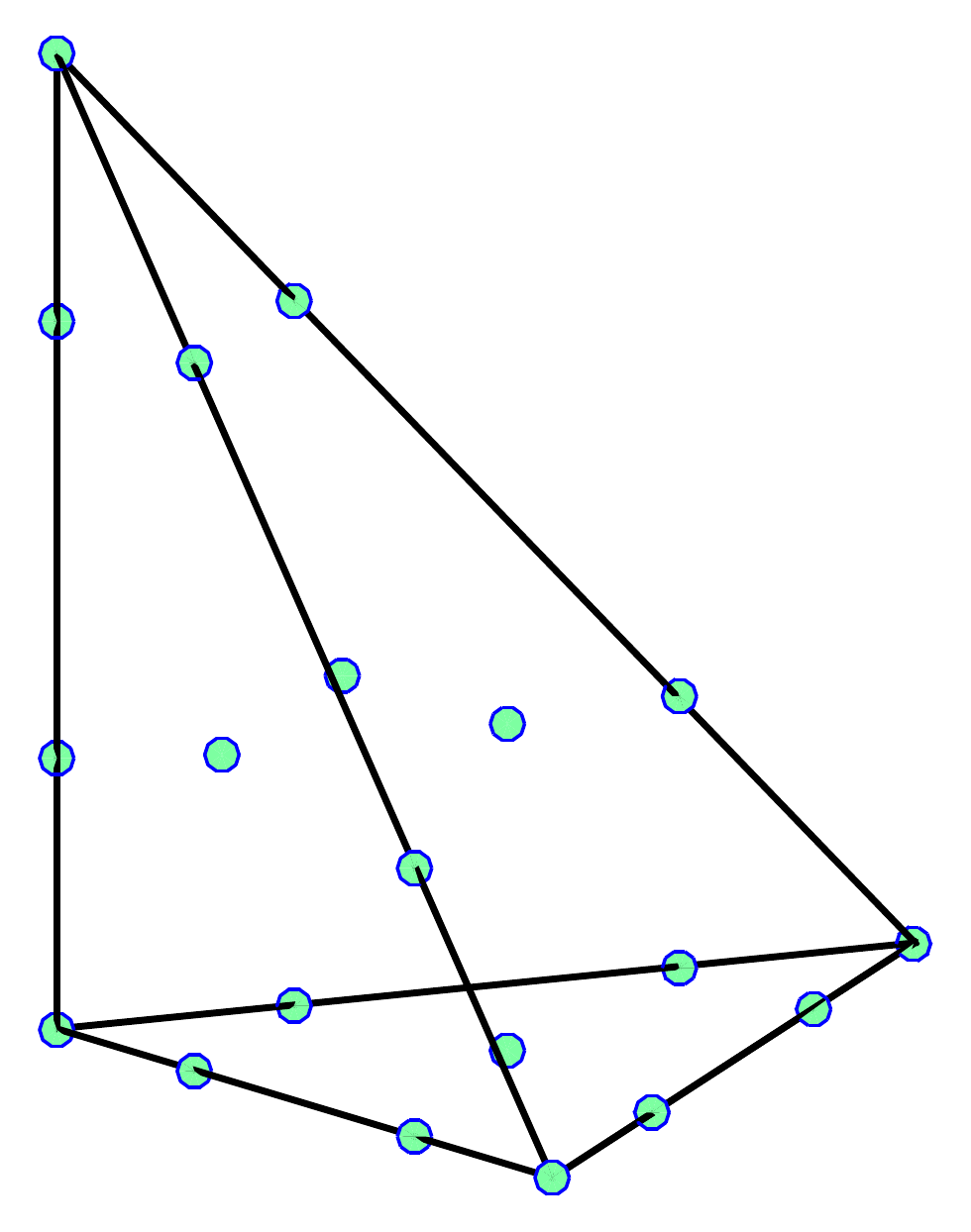}}
\caption{Interpolation points for the wedge and tetrahedron for $N=3$.}
\label{fig:wn}
\end{figure}

\subsection{Geometric factors for vertically mapped wedges}

Because any two vertex-mapped tetrahedra are affinely related, geometric factors for tetrahedra are constant from element to element.  This results in negligible storage costs, as physical matrices are simply scaled combinations of reference mass, derivative, and lift matrices.  However, as noted in \cite{chan2015gpu}, polynomial bases which result in diagonal mass matrices for general vertex-mapped wedges are not currently available.  As a result, the cost of storing wedge mass matrix inverses or lift matrices can be excessive, especially at high orders of approximation.  This is problematic for efficient implementations on many-core architectures (such as GPUs) that typically have very limited memory.  In this work, we consider hybrid tet-wedge meshes containing ``vertically-mapped'' wedges, or wedges which are deformed from their reference configuration along in the vertical (extruded) coordinate.  Restricting to such wedges allows us to greatly decrease memory costs for time-domain DG methods.  
 
 Physical wedges are images of reference wedges under the local mapping $\bm{\Phi}^k(r,s,t) = \sum_{i = 1}^6 \bm{\nu}_i v_i(r,s,t)$, where $\bm{\nu}_i \in \mathbb{R}^3$ are coordinates of the six vertices on a wedge, and $v_i(r,s,t)$ are low-order vertex basis functions
\begin{align*}
v_1(r,s,t) &= (1-r-s)(1-t)\\
v_2(r,s,t) &= r(1-t)\\
v_3(r,s,t) &= s(1-t)\\
v_4(r,s,t) &= (1-r-s)t\\
v_5(r,s,t) &= rt\\
v_6(r,s,t) &= st.
\end{align*}
for reference coordinates $0 \leq t \leq 1$, $0 \leq r, s \leq 1$, and $0 \leq r+s \leq 1$.  We define three wedge ``vertex pairs'' $(\bm{v}_1, \bm{v}_4)$, $(\bm{v}_2, \bm{v}_5)$, and $(\bm{v}_3, \bm{v}_6)$.  These correspond to vertices of the reference wedge which have the same $r,s$ coordinates.  

The Jacobian of this mapping is defined by 
\[
J^k = {\rm det} \LRs{\pd{\bm{\bm{\Phi}^k}}{r}{}, \pd{\bm{\bm{\Phi}^k}}{s}{}, \pd{\bm{\bm{\Phi}^k}}{t}{}}.
\]
If $\bm{\nu}_i$ are taken to be the vertex positions of the reference wedge, $\bm{\bm{\Phi}^k} = \bm{1}$ by the fact that the low order wedge vertex functions $v_i(r,s,t)$ sum to unity.  

In this work, we restrict ourselves to \textit{vertically mapped} wedges: 
\begin{definition}
Let $D^k$ be a wedge element; then, $D^k$ is a vertically mapped wedge if the $x$ and $y$ coordinates are identical for each vertex pair $(\bm{v}_1, \bm{v}_4)$, $(\bm{v}_2, \bm{v}_5)$, and $(\bm{v}_3, \bm{v}_6)$.  
\end{definition}
We note that the definition of vertically mapped wedges allows wedges whose triangular faces are co-planar; however, such elements are not admissible as they have zero volume and thus have $J^k = 0$.  
%We may assume without loss of generality that the first vertex is anchored at $(0,0)$, since this is equivalent to an affine translation of the wedge.  Then, 
%\edit{rewrite for mathematical accuracy.  Make $J$ notation consistent.}

For simplicity, we denote $\pd{r}{x}$ as $r_x$, and similarly for $\pd{s}{x} = s_x, \pd{t}{x} = t_x$, and all other geometric factors.  Then, we have the following Lemma:
\begin{lemma}
\label{lemma:wedge}
For vertically mapped wedges, $J^k$ is constant in the $t$ direction.  Additionally, the geometric factors have the following properties
\begin{itemize}
\item The geometric factors $r_z, s_z = 0$,  
\item The quantities $r_x, r_y, s_x, s_y, \LRp{t_z J^k}$ are constant throughout an element
\item The quantities $t_x J^k, t_y J^k$ vary only linearly in the $t$ coordinate.
\end{itemize}
\end{lemma}
\begin{proof}
We may exploit the fact that the $x,y$-coordinates of different vertex pairs are identical.  For example, the $x$-coordinates of $\bm{v}_1$ and $\bm{v}_4$ are identical, implying that
\[
\bm{v}^x_1 v_1(r,s,t) + \bm{v}^x_4 v_4(r,s,t) = \LRp{\bm{v}^x_1 + \bm{v}^x_4} (1-r-s).  
\]
Similarly, the $x$-coordinates of the $\bm{v}_2, \bm{v}_5$ are identical, as are the $x$-coordinates of $\bm{v}_3,\bm{v}_6$, such that
\[
\bm{v}^x_2 v_2(r,s,t) + \bm{v}^x_5 v_5(r,s,t) = \LRp{\bm{v}^x_2 + \bm{v}^x_5} r, \qquad \bm{v}^x_3 v_3(r,s,t) + \bm{v}^x_6 v_6(r,s,t) = \LRp{\bm{v}^x_3 + \bm{v}^x_6} s.
\]
This implies that $x$ is an affine function of $r,s$, and that $\pd{x}{r}{}, \pd{x}{s}{}$ are constant throughout the wedge.  As the $y$-coordinates of each of these vertex pairs is also identical, the same argument implies that $\pd{y}{r}{}, \pd{y}{s}{}$ are also constant throughout the wedge.  Additionally, since $x,y$ are affine functions of $r,s$ only, this implies $\pd{x}{t}, \pd{y}{t} = 0$.  

The $z$-coordinates of each vertex pair are arbitrary; however, since derivatives of the vertex functions $\pd{v_i}{r}$ and $\pd{v_i}{s}$ are constant in the $r,s$ coordinates, $\pd{z}{r}$ and $\pd{z}{s}$ vary only in the $t$ coordinate.  Similarly, since all derivatives $\pd{v_i}{t}$ are constant in the $t$ coordinate, $\pd{z}{t}$ can vary only in the $r,s$ coordinates.  

Thus, using that $\pd{x}{t} = \pd{y}{t} = 0$, the expression for $J^k$ simplifies to 
\[
J^k = \pd{x}{r}{}\pd{y}{s}{}\pd{z}{t}{} - \pd{x}{s}{}\pd{y}{r}{}\pd{z}{t}{}.
\]
The partial derivatives $\pd{z}{r}, \pd{z}{s}$ are the only factors which vary in $t$; since they do not appear in the above expression, $J$ is constant along the $t$ coordinate.

The formulas for geometric factors \cite{karniadakis2013spectral, hesthaven2007nodal} are also simplified since $x_t, y_t  = 0$  
\begin{align*}
r_x &=  (y_s z_t - z_s y_t)/J^k \qquad r_y = -(x_s z_t - z_s x_t)/J^k \qquad r_z = (x_s y_t - y_s x_t)/J^k\\ 
s_x &= -(y_r z_t - z_r y_t)/J^k \qquad s_y =  (x_r z_t - z_r x_t)/J^k \qquad s_z = -(x_r y_t - y_r x_t)/J^k \\
t_x &=  (y_r z_s - z_r y_s)/J^k \qquad t_y = -(x_r z_s - z_r x_s)/J^k \qquad t_z = (x_r y_s - y_r x_s)/J^k.
\end{align*}
Straightforward computations give the remainder of the lemma.
\end{proof}

\section{Reducing storage costs for nodal DG methods on semi-structured prismatic meshes} 

Since wedges are constructed as the tensor product of a triangle and one-dimensional basis, mass lumping techniques are often employed in the extruded direction \cite{karniadakis2013spectral, kirby2000discontinuous, bergot2013higher} to reduce storage costs.  For a general wedge, mass lumping reduces the cost of explicitly storing $\LRp{\bm{M}^k}^{-1}$ from $O(N^6)$ to $O(N^5)$, and reduces the cost of explicit storage of the lift matrices $\bm{L}^k_f$ (under the use of mass-lumped surface quadratures) from $O(N^5)$ to $O(N^4)$.  However, since mass lumping for wedges corresponds to the use of underintegrated quadratures in integrals, they can introduce aliasing errors which result in non-energy stable DG methods.  This is discussed in more detail in Section~\ref{sec:lump}.  

Instead of employing mass lumping, we seek a way to reduce the storage costs of DG for semi-structured prismatic meshes.  Given Lemma~\ref{lemma:wedge}, it is possible to reduce the storage of exact quadrature DG methods for meshes with vertically mapped wedges.  This method is based on the exact DG formulation (without the use of underintegrated quadratures); however, the implementation of the discrete formulation greatly reduces the storage compared to a straightforward implementation.  

\subsection{Mass lumping and energy stability}
\label{sec:lump}

Mass lumped methods take advantage of the fact that Gauss-Legendre-Lobatto nodes are also quadrature nodes, and evaluate entries of the mass matrix using Gauss-Legendre-Lobatto quadrature in the $t$ coordinate.  The resulting entries of the wedge mass matrix are
\begin{align*}
\edit{\LRp{\bm{M}^k}_{ij,i'j'}} &= \edit{\int_{\widehat{D}} \ell_{ij}\ell_{i'j'} J = \int_{-1}^1 \int_{\widehat{T}} \ell_{ij}\ell_{i'j'} J(r,s,t)}\\
&= \edit{\sum_{m = 0}^N \int_{\widehat{T}} \ell_{ij}(r,s,t_m)\ell_{i'j'}(r,s,t_m) J(r,s,t_m) w_m}\\
&= \edit{\delta_{jj'} \int_{\widehat{T}} \ell_{ij}(r,s,t_j)\ell_{i'j'}(r,s,t_j) J(r,s,t_j) w_j},
\end{align*}
where $\widehat{T}$ is the reference triangle.  This mass-lumping procedure results in wedge mass matrices which are block diagonal, where each block is a square matrix of size $\frac{(N+1)(N+2)}{2}$.  However, since Gauss-Legendre-Lobatto rules are only exact for polynomials up to degree $(2N-1)$, the use of mass lumping introduces quadrature/aliasing errors.  For mass lumped hexes, these errors do not adversely affect energy stability \cite{kopriva2010quadrature}; however, on non-hexahedral elements, these aliasing errors can result in a discrete method which is unstable.  

Proofs of energy stability of time-domain DG methods on non-hexahedral elements typically assume exact quadrature, or require \textit{a priori} stable variational formulations in the presence of quadrature errors \cite{warburton2013low}.  Under-integration and mass lumping with the ``strong'' variational formulation can produce unstable aliasing errors.  Recall the local variational formulation in Section~\ref{sec:form}
\begin{align*}
\int_{D^k} \frac{1}{\kappa}\pd{p}{\tau}{}\phi^- \diff x &= -\int_{D^k} \Div\bm{u}\phi^- \diff x + \int_{\partial D^k} \frac{1}{2}\LRp{\tau_p\jump{p} - \bm{n}\cdot \jump{\bm{u}} }\phi^- \diff x  \\
\int_{D^k} \rho\pd{\bm{u}}{\tau}{}\bm{\psi}^- \diff x &= - \int_{D^k} \Grad p \cdot  \bm{\psi}^- \diff x + \int_{\partial D^k} \frac{1}{2}\LRp{\tau_u \jump{\bm{u}}\cdot \bm{n}^- - \jump{p}}\bm{\psi}^-\cdot \bm{n}^- \diff x.
\end{align*}
If integration by parts holds at the discrete level, the above formulation is discretely equivalent to the ``skew-symmetric'' formulation
\begin{align*}
\int_{D^k} \frac{1}{\kappa}\pd{p}{\tau}{}\phi^- \diff x &= \int_{D^k} \edit{\bm{u} \cdot \Grad} \phi^- \diff x + \int_{\partial D^k} \frac{1}{2}\LRp{\tau_p\jump{p} - 2\bm{n}\cdot \avg{\bm{u}} }\phi^- \diff x  \\
\int_{D^k} \rho\pd{\bm{u}}{\tau}{}\bm{\psi}^- \diff x &= - \int_{D^k} \Grad p \cdot  \bm{\psi}^- \diff x + \int_{\partial D^k} \frac{1}{2}\LRp{\tau_u \jump{\bm{u}}\cdot \bm{n}^- - \jump{p}}\bm{\psi}^-\cdot \bm{n}^- \diff x,
\end{align*}
then energy stability follows by taking $\phi^- = p$ and $\bm{\psi}^- = \bm{u}$ (see, for example \cite{hesthaven2007nodal,chan2016weight1}).  \edit{Pulling the time derivative out of the left hand side term and summing both equations over all elements $D^k$ yields}
\begin{align}
&\edit{\frac{1}{2}\td{}{\tau}\sum_{{D^k} \in \Omega_h} \left\{\int_{D^k} \LRp{ \frac{1}{\kappa} p^2 + \rho \bm{u}^2}\right\}}= \nonumber\\
&\edit{\frac{1}{2}\sum_{D^k\in \Oh} \int_{\partial D^k} \tau_p\jump{p}p^- + \tau_u \jump{\bm{u}}\cdot \bm{n}^- \bm{u}^-\cdot \bm{n}^-  - 2\bm{n}\cdot \avg{\bm{u}} p^- - \jump{p}\bm{u}^-\cdot \bm{n}^- \diff x.}\label{eq:estab}
\end{align}
\edit{The surface terms may be rewritten as sums over unique faces $f\in \Gh$, where $\Gh$ denotes the set of unique faces in $\Oh$.  Expanding out the face integrands in terms of $\bm{n}^-$, $(p^-,\bm{u}^-)$, and $(p^+,\bm{u}^+)$ shows that the terms which are independent of $\tau_p,\tau_u$ cancel.  The remaining penalization terms then reduce to $-\tau_p\jump{p}^2$ and $-\tau_u\jump{\bm{u}_n}^2$ over each face, and inserting them into (\ref{eq:estab}) yields}
\[
\edit{\td{}{\tau}\sum_{{D^k} \in \Omega_h} \left(\int_{D^k} \LRp{ \frac{1}{\kappa} p^2 + \rho \bm{u}^2}\right)  = -\sum_{f \in \Gh} \int_{\partial {f}} \LRp{\tau_p\jump{p}^2 + \tau_u\jump{\bm{u}_n}^2} \leq 0.}
\]
This implies that for $\tau_p, \tau_u \geq 0$, DG methods exhibit dissipative behavior, where unresolvable components of the solution corresponding to high frequency modes are damped \cite{hesthaven2002nodal, hesthaven2004high}.  More precisely, after discretizing in space using DG, the PDE reduces to a system of ODEs for the nodal values $\bm{U}$
\[
\td{\bm{U}}{\tau} = \bm{A}\bm{U},
\]
where we refer to $\bm{A}$ as the DG right-hand side matrix.  For $\tau_p, \tau_u > 0$, we refer to the flux as an \textit{upwind flux}.  For upwind fluxes, the eigenvalues of $\bm{A}$ contain negative real parts, indicating dissipation in time of the solution.  Additionally, for $\tau_p = \tau_u = 0$, the DG formulation reduces to that of a central flux, resulting in a non-dissipative method where the eigenvalues of $\bm{A}$ are purely imaginary.  \edit{In the following experiments, we build $\bm{A}$ for $\rho = c^2 = 1$, homogeneous Dirichlet boundary conditions, and mesh of 16 wedges whose vertices are randomly perturbed in the vertical direction to ensure non-constant $J^k$. }

\begin{figure}
\centering
\subfloat[Upwind flux]{\includegraphics[width=.31\textwidth]{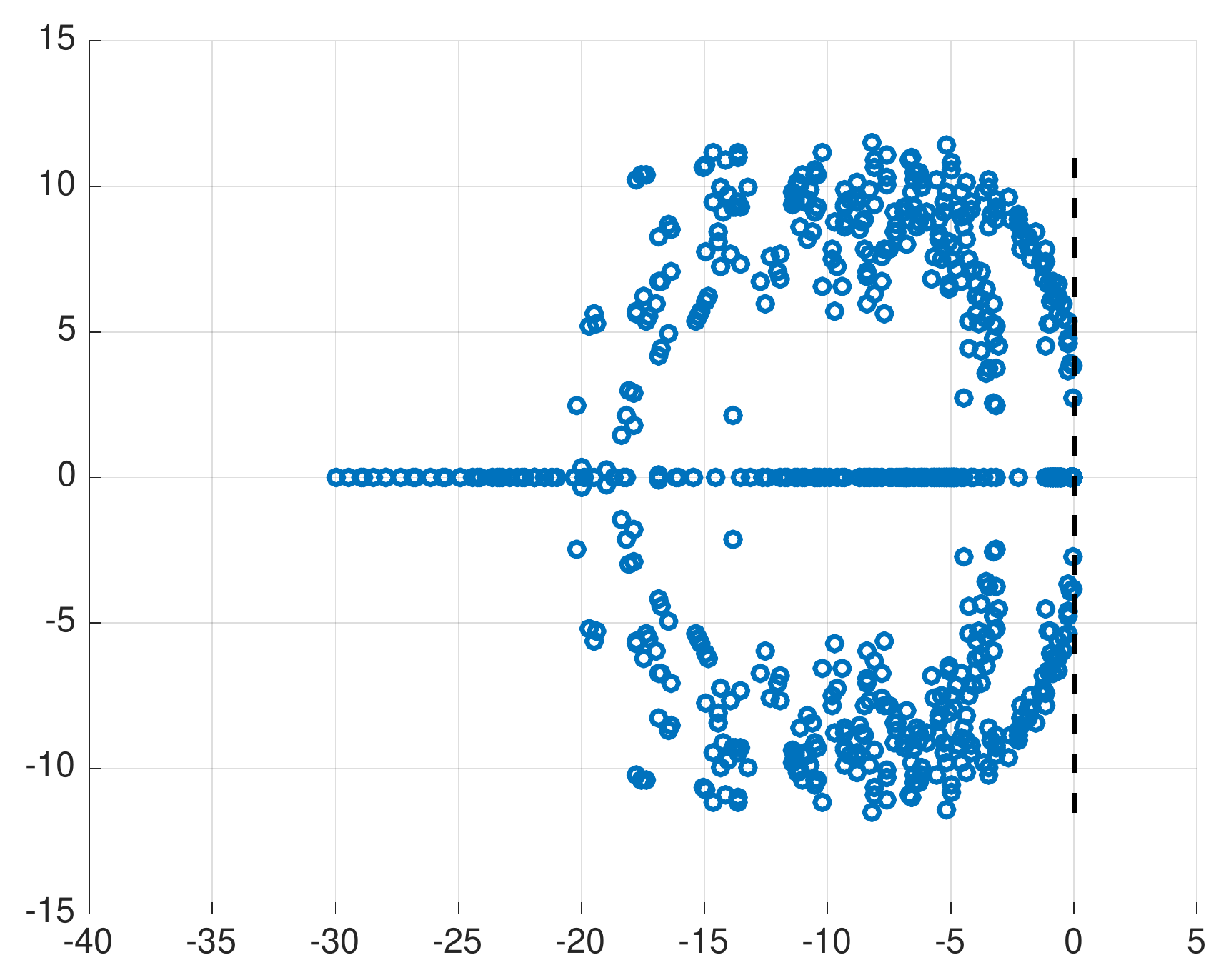}}
\hspace{.5em}
\subfloat[Upwind flux (zoom)]{\includegraphics[width=.31\textwidth]{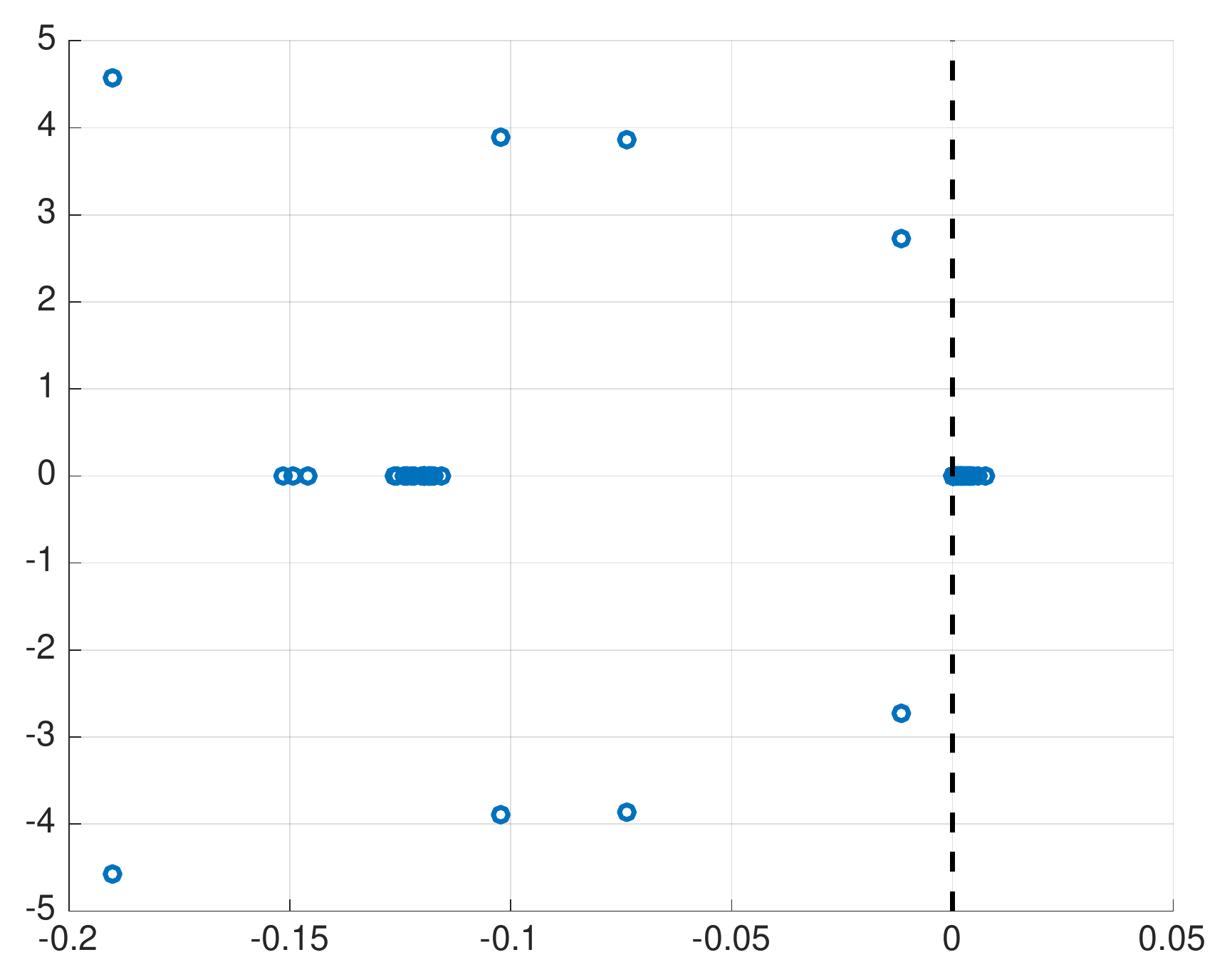}}
\hspace{.5em}
\subfloat[Central flux]{\includegraphics[width=.31\textwidth]{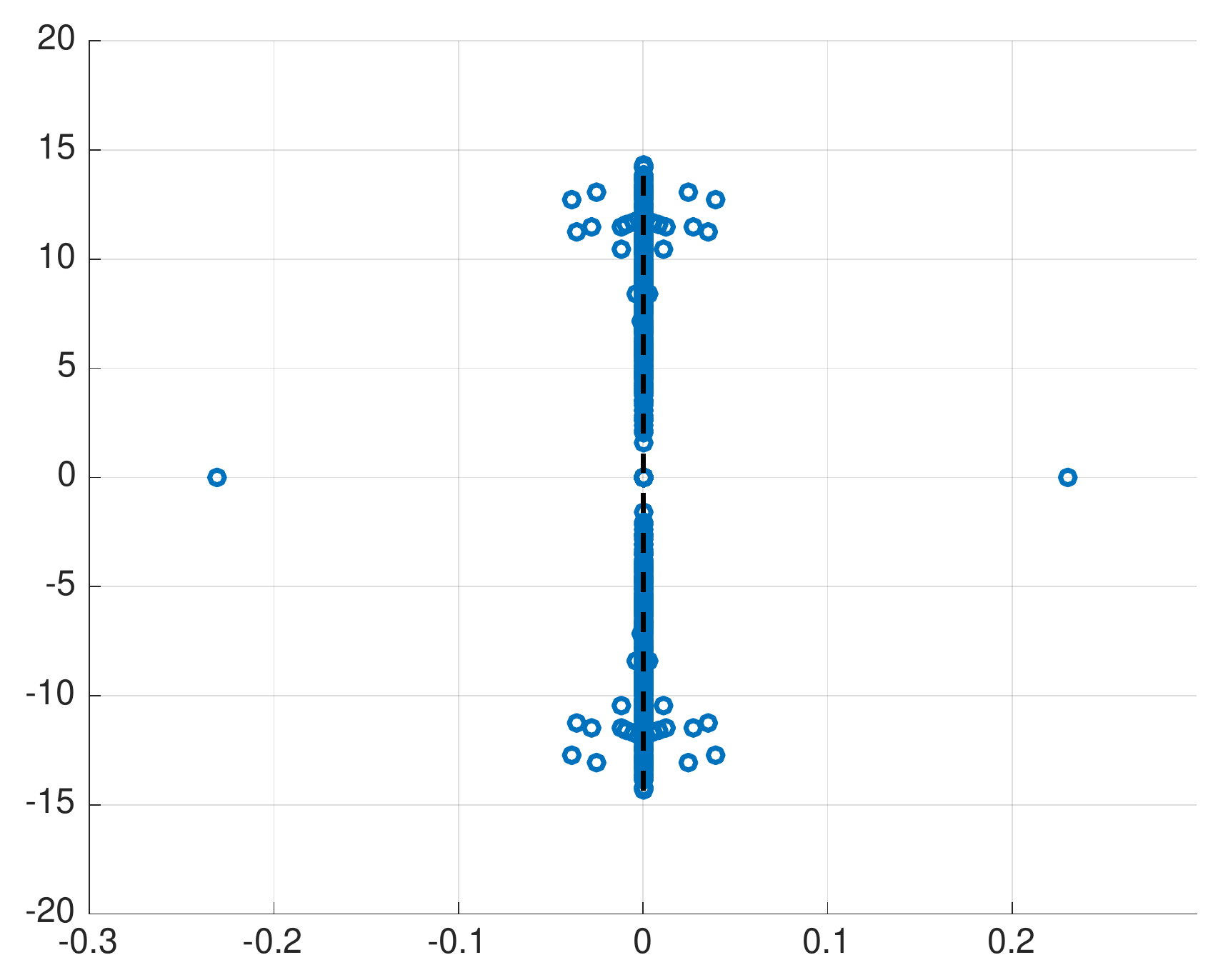}}
\caption{Eigenvalues for a mass-lumped \edit{DG} discretization \edit{on wedges} with $N=2$.  Unlike non-mass lumped discretizations, there exist  eigenvalues with positive real parts, implying that energy stability is lost due to errors introduced by underintegration.  }
\label{fig:eigs}
\end{figure}

For Discontinuous Galerkin spectral element methods (DG-SEM), which utilize mass lumping on quadrilateral or hexahedral elements, integration by parts holds under Gauss-Legendre-Lobatto quadrature and the strong and skew-symmetric formulation are discretely equivalent \cite{kopriva2010quadrature, gassner2013skew}.  This equivalence relies on the diagonal nature of the mass matrix and the exactness of Gauss-Legendre-Lobatto quadrature for degree $(2N-1)$ polynomials.  Unfortunately, the analogous statement does not hold true for mass lumped wedges.  Figure~\ref{fig:eigs} shows the spectra of the DG right hand side operator resulting from a discretization of the strong formulation using mass lumping on a mesh with vertically mapped wedges.  For both upwind and central fluxes, eigenvalues with real positive part are present, implying that the formulation is not energy stable.  In contrast, employing a non-mass lumped discretization is equivalent to using an exact quadrature rule, such that continuous integration by parts implies discrete equivalence of the strong formulation and energy-stable skew-symmetric formulation.  This is reflected in Figure~\ref{fig:eigsex}, which shows the eigenvalues of the DG operator under our method with exact quadrature.  

\begin{figure}
\centering
\subfloat[Upwind flux]{\includegraphics[width=.31\textwidth]{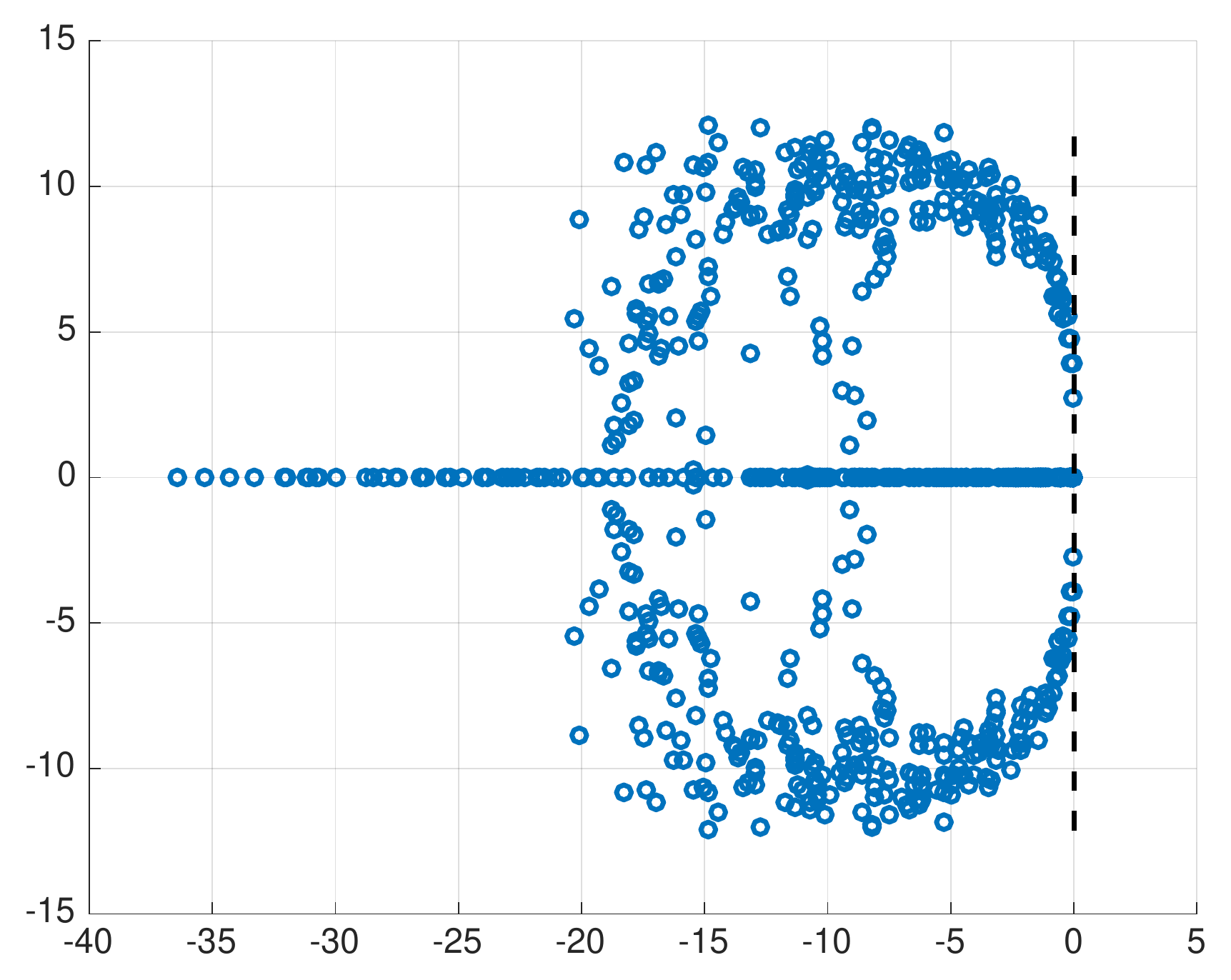}}
\hspace{.5em}
\subfloat[Upwind flux (zoom)]{\includegraphics[width=.31\textwidth]{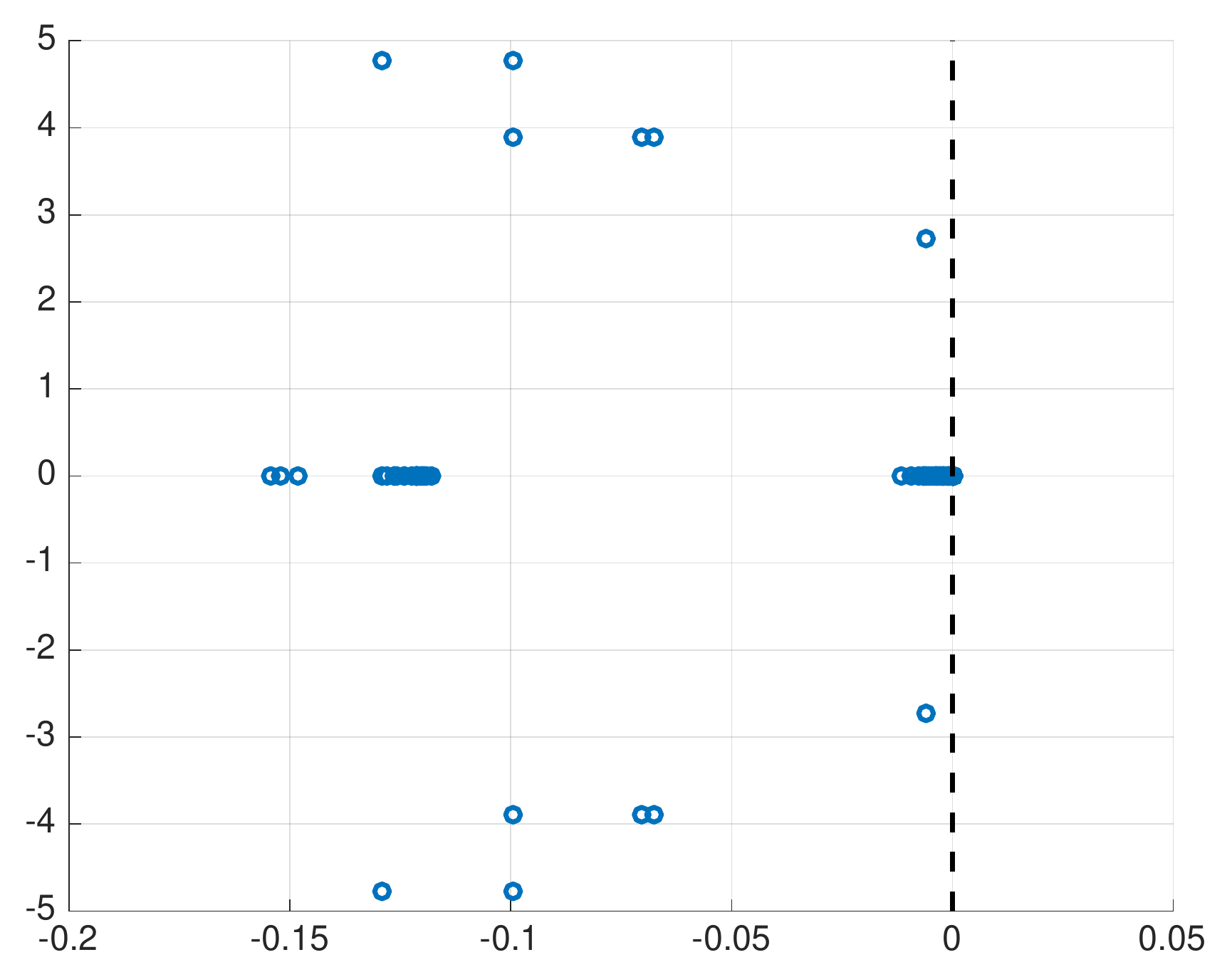}}
\hspace{.5em}
\subfloat[Central flux]{\includegraphics[width=.31\textwidth]{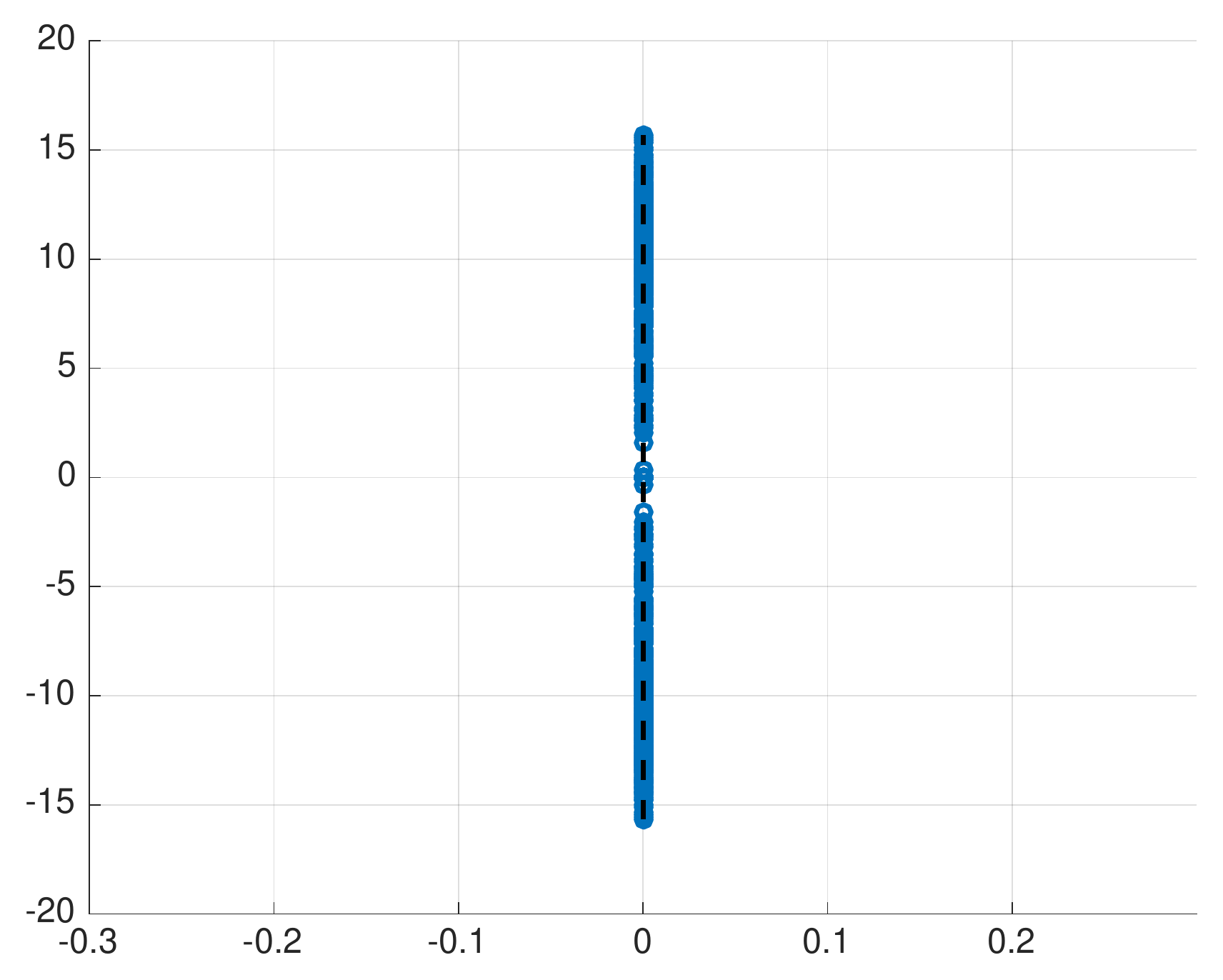}}
\caption{Eigenvalues for the exact quadrature \edit{DG} discretization \edit{on wedges} with $N=2$.  All eigenvalues have non-positive real part \edit{(up to machine precision)}, implying that the formulation is energy stable.}
%\caption{Eigenvalues for the exact quadrature DG discretization on wedges with $N=2$.  All eigenvalues have non-positive real part (up to machine precision), implying that the formulation is energy stable.}
\label{fig:eigsex}
\end{figure}

\subsection{Mass, differentiation, and lift matrices}
\label{sec:mats1}

Since mass lumping on wedges produces DG methods which are not energy stable, we seek to reduce the storage costs for exact quadrature DG methods on wedge meshes, taking advantage of the fact that for vertically mapped wedges, $J^k$ is constant in $t$.  Since wedge nodal basis functions can be defined as the tensor product of triangular and one-dimensional nodal basis functions, the mass matrix over an element $D^k$ has entries
\begin{align*}
\edit{\LRp{\bm{M}^k}_{ij,i'j'} = \int_{\widehat{D}} \ell_{ij}\ell_{i'j'} J = \int_{-1}^1 \ell^{\rm 1D}_j(t) \ell^{\rm 1D}_{j'}(t) \int_{\widehat{T}} \ell^{\rm tri}_{i}\ell^{\rm tri}_{i'} J(r,s)}.
\end{align*}
This implies that $\bm{M}^k$ is expressible as
\[
\bm{M}^k = \edit{\bm{M}^{{\rm tri},k}} \otimes \edit{\widehat{\bm{M}}^{\rm 1D}}
\]
where the entries of the triangle mass matrix $\edit{\bm{M}^{{\rm tri},k}}$ for a given element are
\[
\LRp{\edit{\bm{M}^{{\rm tri},k}}}_{ij} = \int_{\widehat{T}} \ell^{\rm tri}_{i}(r,s)\ell^{\rm tri}_{j}(r,s) J(r,s).
\]
\edit{We include superscripts $k$ in $\bm{M}^{{\rm tri},k}$ to emphasize that the triangle mass matrices are distinct from element to element.}

Face mass matrices under vertically mapped wedges also possess a Kronecker product structure.  The entries of face mass matrices for a face $f$ of an element $D^k$ are given by 
\[
\edit{\LRp{\bm{M}^k_f}_{ij,i'j'} = \int_{f} \ell_{ij}\ell_{i'j'} J^k_f,}
\] 
where $J^k_f$ is the Jacobian of the mapping from a reference triangle or quadrilateral to the physical triangular/quadrilateral face of an element.  For triangular faces, $J^k_f$ is constant, since any two planar triangles are affinely related.  
%Triangular face mass matrices are precomputed using triangle quadratures which are numerically exact for polynomials of degree $(2N+1)$ \cite{xiao2010quadrature}, though these quadratures are not explicitly used in time-evolution steps.  
Since wedge nodes include nodes on each of the triangular faces, entries of the triangular face mass matrix \edit{$\LRp{\bm{M}_f^k}_{ij,i'j'}$ are nonzero only for nodes where $j,j' = 0$ or $j,j' = N$.}  
\[
\bm{M}_f^k = J^k_f\widehat{\bm{M}}^{\rm tri} \otimes \bm{e}
\]
where $\widehat{\bm{M}}^{\rm tri}$ is the reference triangle mass matrix, and $\bm{e} \in \mbb{R}^{N+1}$ is the first canonical basis vector for the triangle face corresponding to $t=-1$ and the last canonical basis vector for the triangular face corresponding to $t=1$, respectively.  

For general quadrilateral faces, $J^k_f$ is not constant.  However, for vertically mapped wedges, quadrilateral vertices are displaced only in the vertical direction, and quadrilateral faces are planar with constant normals.  Using similar arguments as in Lemma~\ref{lemma:wedge}, it is possible to show that $J^k_f$ is constant in the extruded (vertical) $t$ coordinate, and varies linearly in the horizontal coordinate.  %An exact surface quadrature over quadrilateral faces can be constructed as the tensor product of one-dimensional degree $N$ Gauss-Legendre quadrature.  
The mass matrix for a quadrilateral face also possesses a Kronecker product structure
\[
\bm{M}_f^k = \edit{\tilde{\bm{M}}^{{\rm tri},k}_{\rm edge}} \otimes \edit{\widehat{\bm{M}}^{\rm 1D}}
\]
where $\edit{\tilde{\bm{M}}^{{\rm tri},k}_{\rm edge}}$ is the face (edge) mass matrix for a triangle, which is computed using the restriction of the quadrilateral face Jacobian $J^k_f$ to the edge
\[
\LRp{\edit{\tilde{\bm{M}}^{{\rm tri},k}_{\rm edge}} }_{ij} = \int_{-1}^1 \ell^{\rm tri}_i \ell^{\rm tri}_j J^k_f.
\]
Since $J^k_f$ varies along the edge,  $\edit{\tilde{\bm{M}}^{{\rm tri},k}_{\rm edge}}$ is distinct from face to face.  

%For both mass and face mass matrices, the only difference between a mass lumped and non-mass lumped discretization for vertically mapped wedges is the definition of $\edit{\widehat{\bm{M}}^{\rm 1D}}$.  Under mass lumping, $\edit{\widehat{\bm{M}}^{\rm 1D}}$ is diagonal, while it is dense for non-mass lumped discretizations.  
The tensor-product nature of the face mass matrices simplifies the form of the lift matrix.  For a triangular face $f$, the lift matrix $\bm{L}^k_f$ reduces to 
\[
\bm{L}^k_f = \LRp{\bm{M}^k}^{-1} \bm{M}^k_f = J^k_f\LRp{\edit{\bm{M}^{{\rm tri},k}}}^{-1}\widehat{\bm{M}}^{\rm tri} \otimes \LRp{\edit{\widehat{\bm{M}}^{\rm 1D}}}^{-1} \bm{e},
\]
while the lift matrix for a quadrilateral face $f$ is
\[
\bm{L}^k_f = \LRp{\edit{\bm{M}^{{\rm tri},k}}}^{-1}\edit{\tilde{\bm{M}}^{{\rm tri},k}_{\rm edge}} \otimes \bm{I}.
\]
This implies that, for vertically mapped wedges, lift matrices for quadrilateral faces are always block diagonal.  

We refer to the matrix $\LRp{\edit{\bm{M}^{{\rm tri},k}}}^{-1}\widehat{\bm{M}}^{\rm tri}$ as the triangular lift $\edit{\bm{L}^{{\rm tri},k}}$.  Both lift matrices for triangular and quadrilateral faces are tensor products of the triangular lift matrix and a one-dimensional matrix.  The triangular lift matrix also appears again in the expression for differentiation matrices.

Since nodal basis functions for the wedge are constructed in a tensor-product fashion, reference differentiation matrices possess a Kronecker product structure.  We define $\bm{D}_r^{\rm tri}, \bm{D}_s^{\rm tri}$ to be differentiation matrices for nodal basis functions defined on the triangle (using Warp and Blend nodes for the face of a tetrahedron).  Likewise, we define $\bm{D}_t^{\rm 1D}$ as the differentiation matrix for a Lagrange basis at degree $N$ Gauss-Legendre-Lobatto nodes in one dimension.  Then, the reference wedge differentiation matrices are 
\[
\bm{D}_r = \bm{D}_r^{\rm tri} \otimes \bm{I}^{\rm 1D}, \qquad \bm{D}_s = \bm{D}_s^{\rm tri} \otimes \bm{I}^{\rm 1D}, \qquad \bm{D}_t = \bm{I}^{\rm tri} \otimes \bm{D}_t^{\rm 1D},
\]
where $\bm{I}^{\rm tri}, \bm{I}^{\rm 1D}$ are square identity matrices of size $(N+1)(N+2)/2$ and $(N+1)$, respectively.  

Weak differentiation matrices also exhibit a Kronecker product structure for vertically mapped wedges.  Lemma~\ref{lemma:wedge} gives that $r_x, s_x, r_y, s_y$ are constant within an element, such that
\[
\edit{\LRp{\bm{S}^k_x}_{ij,i'j'} = \int_{\widehat{D}} \pd{\ell_{i'j'}}{x}\ell_{ij} J^k = 
r_x\int_{\widehat{D}} \pd{\ell_{i'j'}}{r} \ell_{ij} + s_x\int_{\widehat{D}}\pd{\ell_{i'j'}}{s}\phi_{ij} J^k  + \int_{\widehat{D}} \pd{\ell_{i'j'}}{t} \ell_{ij} t_x J^k}.
\]
Noting from Lemma~\ref{lemma:wedge} that $t_xJ^k$ is a polynomial of degree 1 in $t$ and that \edit{$\pd{\ell_{i'j'}}{t}$} is a polynomial of degree $(N-1)$, their product is a polynomial of degree $N$.  Since multiplication by the mass is equivalent to integration of a degree $N$ polynomial against elements of the basis $\ell_{ij}$, we have that
\[
\edit{\int_{\widehat{D}} t_x J^k \pd{\ell_{i'j'}}{t} \ell_{ij} = \LRp{ {\edit{\bm{M}^{{\rm tri},k}}} \otimes \edit{\widehat{\bm{M}}^{\rm 1D}} {\rm diag}(t_xJ^k) \bm{D}^{\rm 1D}_t}_{ij,i'j'},}
\]
where ${\rm diag}(t_xJ^k)$ is a diagonal matrix of $t_x J^k$ at Gauss-Legendre-Lobatto nodes in the $t$ coordinate.  The same argument can be made for $\bm{S}^k_y$, implying that 
\begin{align*}
\bm{D}_x^k &= \LRp{r_x \bm{D}_r^{\rm tri} + s_x \bm{D}_s^{\rm tri}} \otimes \bm{I}^{\rm 1D} + \edit{\bm{L}^{{\rm tri},k}} \otimes {\rm diag}(t_x J^k) \bm{D}^{\rm 1D}_t \\
\bm{D}_y^k &= \LRp{r_y \bm{D}_r^{\rm tri} + s_y \bm{D}_s^{\rm tri}} \otimes \bm{I}^{\rm 1D} + \edit{\bm{L}^{{\rm tri},k}} \otimes {\rm diag}(t_y J^k) \bm{D}^{\rm 1D}_t.
\end{align*}
The differentiation matrix $\bm{D}^k_z$ also possesses a tensor product structure; additionally, since $r_z, s_z = 0$ and $t_z$ is constant for vertically mapped wedges, the expression for $\bm{D}^k_z$ simplifies to
\[
\bm{D}^k_z = \LRp{\edit{\bm{M}^{{\rm tri},k}}}^{-1}\widehat{\bm{M}}^{\rm tri} \otimes t_zJ^k \bm{D}^{\rm 1D}_t = \edit{\bm{L}^{{\rm tri},k}} \otimes t_zJ^k \bm{D}^{\rm 1D}_t.  
\]
If the mapping is affine, the triangular lift matrix $\edit{\bm{L}^{{\rm tri},k}}$ reduces to $J^k \bm{I}^{\rm tri}$, weak differentiation matrices reduce to scalings of nodal differentiation matrices \cite{gandham2015high}.  

Finally, since the method in this paper uses exact quadrature, \edit{all the terms in the DG formulation (\ref{eq:form}) are computed exactly} and theoretical results for standard DG methods are automatically inherited, such as high order accuracy and energy stability \cite{brenner2007mathematical, hesthaven2007nodal}.

\subsection{Reduced storage costs}

For general vertex-mapped hexahedra, pyramids, and tetrahedra, there exist bases or stable mass-lumped approximations which yield low-storage mass matrices \cite{karniadakis2013spectral, chan2015orthogonal}.  However, unlike other element types, time domain DG methods on wedges (using polynomial bases) require explicit storage of matrices for each element.  Naive storage of elemental matrices can increase the available memory required from $O(N^3)$ to $O(N^6)$ for each element in three dimensions.  These large storage costs can greatly limit the admissible problem sizes on GPUs and modern accelerator architectures, which tend to possess limited on-device memory.  There does not appear to be a polynomial basis which yields (for non-affine wedges) lift matrices which requires $O(N^3)$ or less storage per element \cite{chan2015gpu}.   However, by assuming vertically mapped wedges, these storage costs may be reduced to a more acceptable level.  

Restricting to vertically mapped wedge elements exposes a Kronecker product structure in both differentiation and lift matrices involving the triangular lift matrix $\edit{\bm{L}^{{\rm tri},k}}  = \LRp{\edit{\bm{M}^{{\rm tri},k}}}^{-1}\widehat{\bm{M}}^{\rm tri}$.  Since $\LRp{\edit{\bm{M}^{{\rm tri},k}}}^{-1}$ varies depending on the local geometric mapping, we explicitly store $\edit{\bm{L}^{{\rm tri},k}}$  in order to avoid solving a matrix equation.  Storage of $\edit{\bm{L}^{{\rm tri},k}}$ requires $O(N^4)$ storage per element.  The lift matrix for quadrilateral faces of the wedge also requires the matrices $ \LRp{\edit{\bm{M}^{{\rm tri},k}}}^{-1}\edit{\tilde{\bm{M}}^{{\rm tri},k}_{\rm edge}}$, for which explicit storage requires only $O(N^3)$ storage per element.  Thus, while we have increased the asymptotic storage cost to $O(N^4)$, we only store one additional matrix of this size per element.  

\begin{figure}
\centering
\includegraphics[width=.5\textwidth]{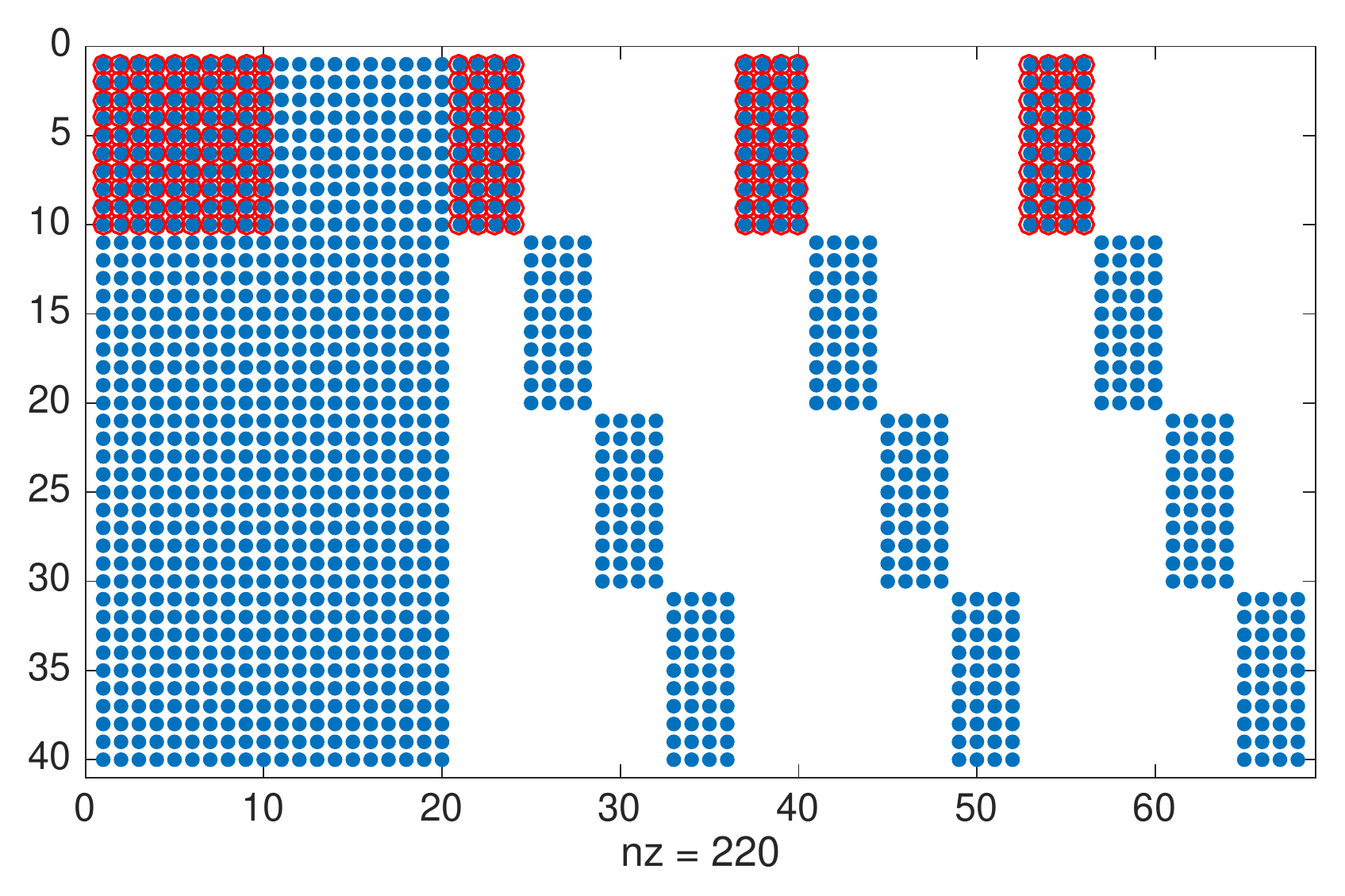}
\caption{\edit{Spy plot of a} lift matrix for a vertically mapped wedge.  Explicitly stored entries are circled in red.  \edit{The axes correspond to row and column indices.}}
\label{fig:lift}
\end{figure}

Figure~\ref{fig:lift} shows the storage required for the lift matrix as compared to the storage of the full lift matrix.  \edit{The matrix shown is a column-wise concatenation of lift matrices for both triangular and quadrilateral faces.  The first $(N+1)(N+2)/2$ columns correspond to the lift matrix for the bottom triangular face, while the next $(N+1)(N+2)/2$ columns correspond to the lift matrix for the top triangular face. The remaining three blocks of $(N+1)^2$ columns correspond to lift matrices for each of the three quadrilateral faces of the wedge.}
Interestingly, one can show that, for vertically mapped wedges, lift matrices under mass lumping have the same storage requirements.  %implying that there is little advantage to employing mass lumping for the wedge meshes considered in this work.  

\section{Hybrid tet-wedge meshes for complex interfaces}
\label{sec:mesh}

We are interested in resolving geometries which contain complex interfaces with separate distinct layers, such as those found in geophysical applications (ocean-earth interfaces and layers of strata).   The approach taken in this work is to develop meshes consisting of wedge elements which conform to interfaces.  \edit{These approaches are also applicable to high order models of atmospheric or coastal flow, which often represent earth-air or ocean-air interfaces based on unstructured meshes of wedges \cite{wan2013icon,casulli2000unstructured}}.  Other techniques for representing complex interfaces include coordinate transformations, such as the sigma transform \cite{phillips1957coordinate}. Typical applications of the sigma coordinate method map a domain with smooth boundaries to a Cartesian domain, where the solution can be determined by solving a modified PDE using numerical methods for structured grids.  In contrast, directly approximating the domain using semi-structured prismatic meshes can handle more general geometries, including non-smooth interfaces.  Coupling to tetrahedral meshes is also made simpler through the use of vertically mapped wedge meshes rather than sigma transformations. 

Constructing multi-layered wedge meshes can be done in layer-by-layer fashion.  We assume that a single layer domain $\Omega_L$ is defined as
\[
\Omega_L = \LRc{ \LRp{x,y} \in [a,b]\times[c,d], \quad z_{\rm bottom}(x,y) \leq z \leq z_{\rm top}(x,y)},
\]
where $z_{\rm top}(x,y), z_{\rm bottom}(x,y)$ are surfaces defining the top and bottom boundaries of the layer.  A surface triangulation is then constructed which resolves both top and bottom layers.  The final step is to produce a quasi-uniform mesh which resolves the interior of the layer.  To do so, we construct a single-layer reference wedge mesh by extruding the surface triangulation and interpolate the surfaces $z_{\rm top}(x,y)$ and $z_{\rm top}(x,y)$ with the vertices at the top and bottom of the reference wedge mesh, respectively.  This produces a linear interpolant between the top and bottom surfaces, and evaluating this interpolant at a point $z$ and at the $(x,y)$ coordinates of the surface triangulation vertices produces the coordinates of an intermediate layer.  Evaluating this interpolant at equispaced points in the vertical coordinate of the reference wedge mesh produces multiple layers, which may then be connected to produce a quasi-uniform wedge mesh which conforms to $z_{\rm top}(x,y)$ and $z_{\rm bottom}(x,y)$.  An example of such a wedge mesh is shown in Figure~\ref{fig:wavy_top}.

\begin{figure}
\centering
\includegraphics[width=.55\textwidth]{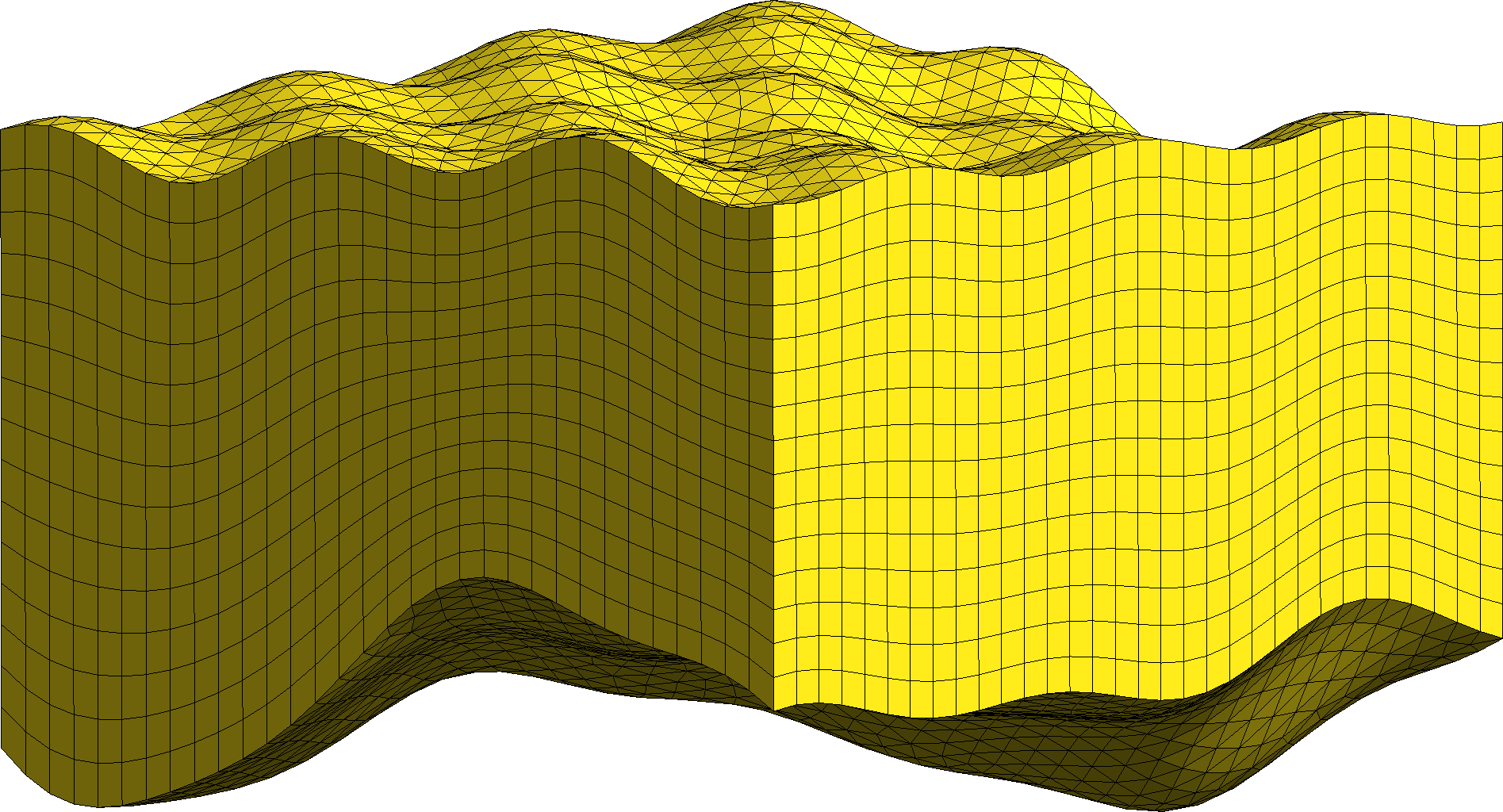}
\caption{Example of meshing a layer defined by two arbitrary surfaces using wedges.}
\label{fig:wavy_top}
\end{figure}

The construction of meshes for layered media can also be done using extrusions of quadrilaterals, resulting in more efficient hexahedral grids.  However, an advantage of wedge meshes is their ability to conform to surfaces of tetrahedral meshes.  This is useful in some geophysical applications, where meshes which conform to interfaces in high contrast media are desirable for capturing sharp transitions between areas of differing wavespeed.  In these cases, it is not typically possible to partition strata of the earth into distinct layers, requiring instead more flexible tetrahedral mesh generation algorithms to produce interface-conforming meshes.  However, wedge meshes may be still be used to capture earth-ocean interfaces or layered structures if present, while tetrahedral meshes are used to capture non-layered high contrast interfaces within the earth. Figure~\ref{fig:hybrid} shows an example of such a mesh, constructed using the above procedure for layered wedge meshes and the Tetgen mesh generation package \cite{si2015tetgen}.  

\begin{figure}
\centering
\subfloat{\includegraphics[width=.4\textwidth]{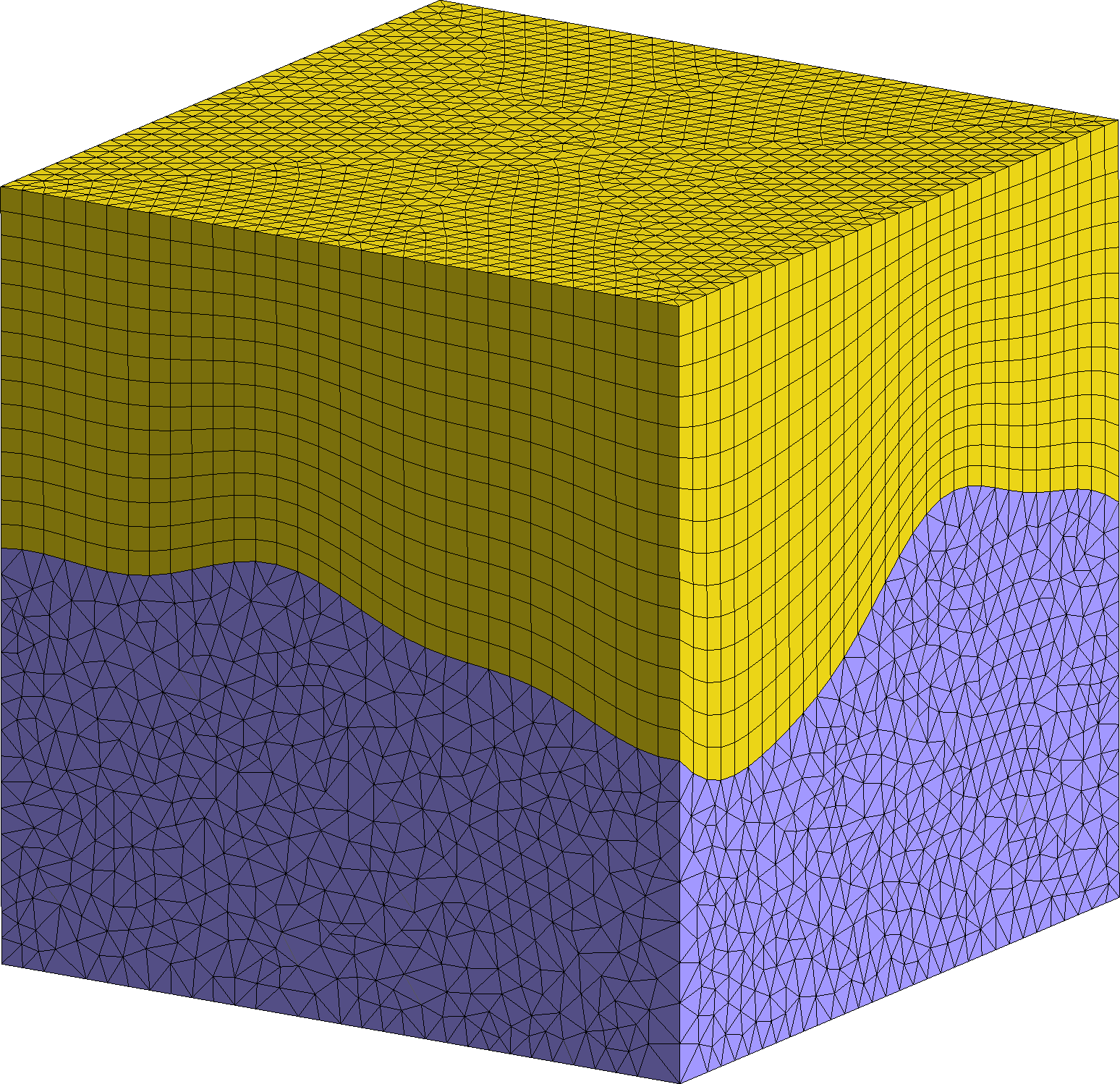}}
\hspace{2em}
\subfloat{\includegraphics[width=.4\textwidth]{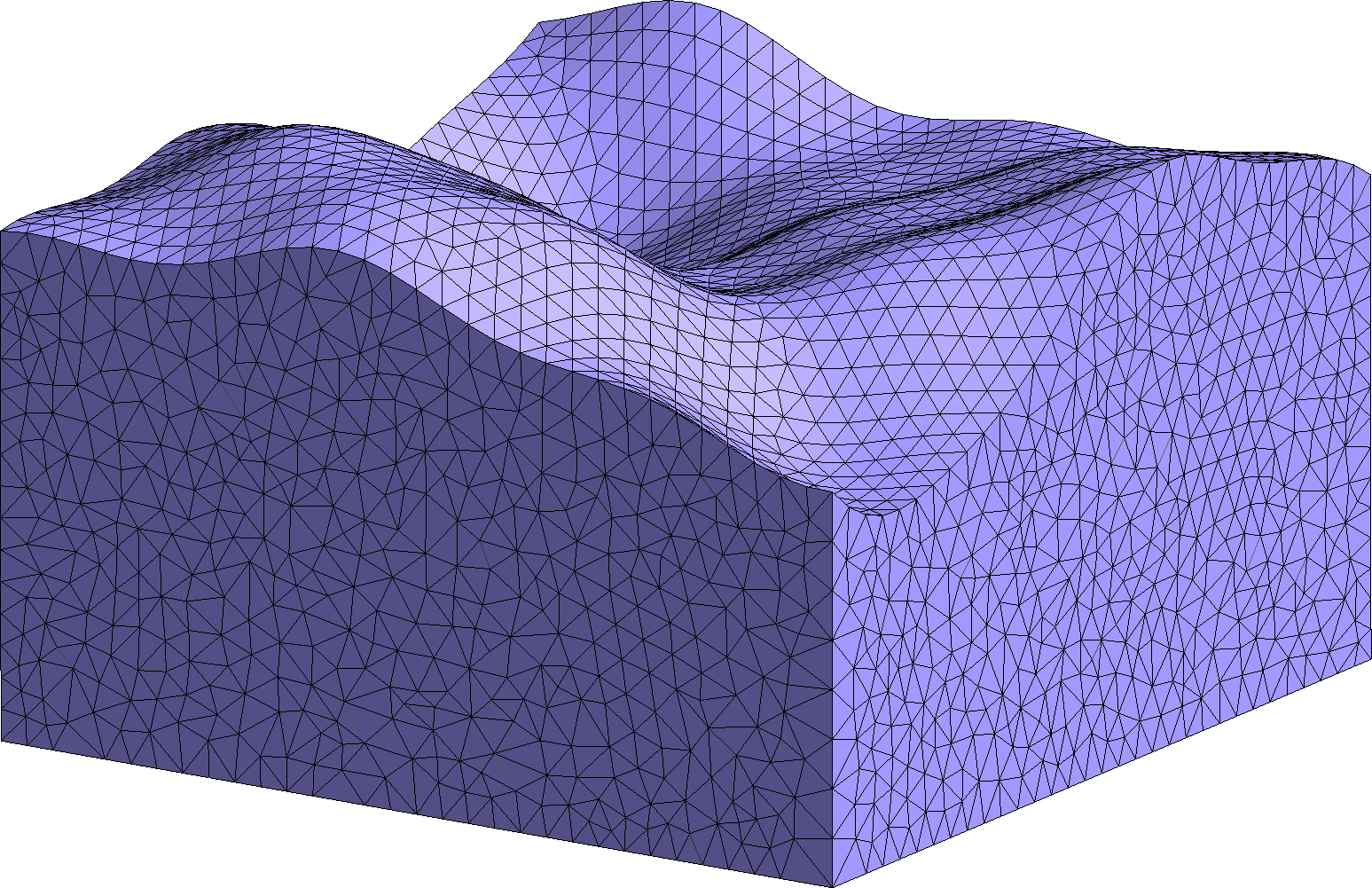}}
\caption{Hybrid wedge-tet meshes for capturing interfaces.}
\label{fig:hybrid}
\end{figure}

\section{Numerical experiments}
\label{sec:num}
In the following sections, we present numerical experiments which confirm the high order accuracy and efficiency of the presented methods on many-core architectures.  \edit{We also present experiments which highlight differences between between different approaches to high order DG methods on hybrid meshes.}  

A GPU-accelerated high order DG method on hybrid meshes was presented in \cite{chan2015gpu}; however the method did not inherit several key features of high order nodal DG methods on tetrahedral meshes.  First, the use of mass lumping and Gauss-Legendre-Lobatto underintegration was shown to decrease convergence rates by half an order in comparison to the use of full quadrature.  Secondly, due to the use of a rational Low-Storage Curvilinear DG (LSC-DG) basis \cite{warburton2010low,warburton2013low}, high order convergence rates were observed but were contingent on refinement strategies resulting in asymptotically affine elements.  The use of LSC-DG reduces storage costs on general vertex-mapped wedges, but results in a method which may not be robust to mesh perturbation, implying that high order rates of convergence could be lost for non-nested or unstructured mesh refinement.  Finally, the use of LSC-DG necessitated the explicit use of quadrature on faces for energy stability, and requires the evaluation of solutions at quadrature points on triangular and quadrilateral faces.  This procedure introduces additional computational work with $O(N^5)$ complexity compared to quadrature-free nodal DG methods.  

The method presented addresses each of these issues for vertically mapped wedges, resulting in a quadrature-free nodal DG method which is both high order accurate method and robust to vertical mesh perturbation.  Discrete timestep restrictions were determined using geometric factors and constants in trace and Markov inequalities \cite{chan2015gpu}.  Experiments were performed on an Nvidia GTX 980 GPU with a cross-platform implementation leveraging the OCCA framework \cite{medina2014occa}.  

\subsection{Convergence tests}
\label{sec:conv}

A significant advantage of high order approximations is that, compared to a low order approximation, far fewer degrees of freedom are required to approximate a sufficiently regular function to a given tolerance.  This is reflected in the convergence rates for high order methods: for a degree $N$ order polynomial approximation, the $L^2$ error converges with rate $O\LRp{h^{N+1}}$, where $h$ is the mesh size.  DG methods with upwind numerical fluxes converge at this optimal rate on a large class of meshes \cite{cockburn2008optimal}, though in general one can expect only a $O\LRp{h^{N+1/2}}$ rate of convergence \cite{johnson1986analysis}.  

It is also possible to diminish the convergence rate of DG methods through underintegration.  For example, mass-lumped DG-SEM methods are observed to converge in the $L^2$ norm at rate $O\LRp{h^{N+1/2}}$ \cite{wilcox2010high,chan2015gpu}.  This is observed to improve to $O\LRp{h^{N+1}}$ if exact quadrature rules are used, at the cost of introducing an extra interpolation step to compute solution values on faces.  

In contrast, the reduced storage DG method in this work avoids both underintegration and extra interpolation steps.  Figure~\ref{fig:rates} shows the convergence of the $L^2$ error for an exact pressure solution 
\[
p(x,y,z,t) = \cos\LRp{\frac{\pi x}{2}}\cos\LRp{\frac{\pi y}{2}}\cos\LRp{\frac{\pi z}{2}}\cos\LRp{\frac{\sqrt{3}\pi t}{2}}
\]
under mesh refinement, and optimal rates of convergence (Table~\ref{table:rates}) are observed.  

\begin{figure}
\centering
\subfloat{\includegraphics[width=.4\textwidth]{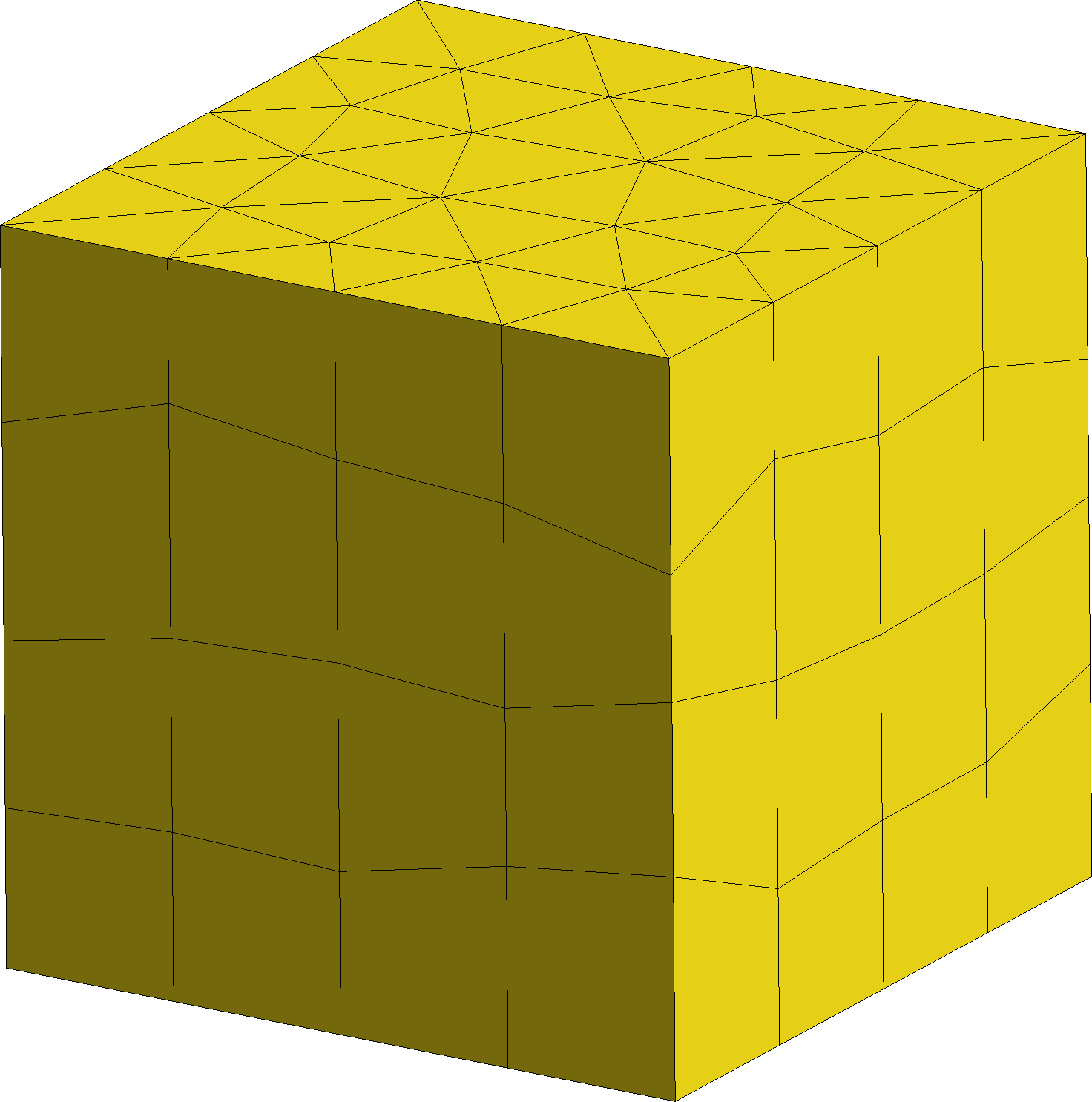}}
\hspace{2em}
\subfloat{\includegraphics[width=.4\textwidth]{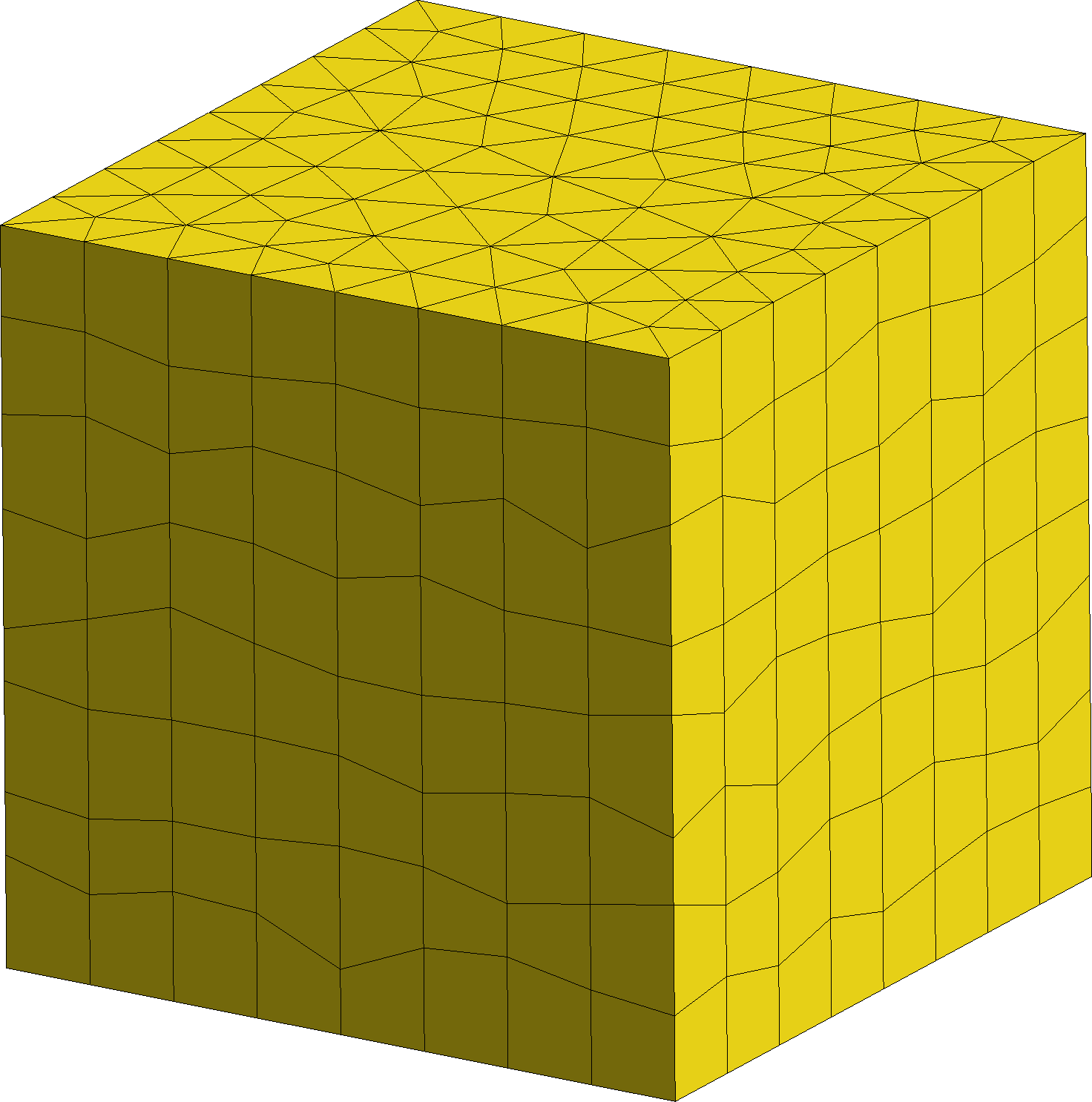}}
\caption{Perturbed vertically mapped wedge meshes used in unstructured mesh refinement tests.}
\label{fig:perturb}
\end{figure}

To verify the robustness of our method \edit{with respect to} unstructured mesh refinement, we examine convergence also on a sequence of vertically mapped perturbed wedge meshes.  Each wedge mesh is generated by extruding an unstructured surface triangulation vertically into several layers, then perturbing the vertical coordinates of each wedge randomly (as shown in Figure~\ref{fig:perturb}).  Optimal rates of convergence of $L^2$ error are observed (Figure~\ref{fig:rates} and Table~\ref{table:rates}) for both structured and unstructured mesh refinement.  

We also compare the $L^2$ error of our limited storage implementation of standard DG to that of LSC-DG on each set of meshes in Figure~\ref{fig:rates}.  On affine meshes, LSC-DG reduces to a standard DG method, and we observe optimal convergence for both methods.  The error for LSC-DG is slightly less than that for DG at lower orders, due to the fact that the initial condition is computed as an $L^2$ projection instead of through Lagrange interpolation (though by $N=3$, this difference is negligible).  For perturbed meshes, $L^2$ errors for LSC-DG begin to stall at small mesh sizes $h$, due to the fact that under perturbed mesh refinement, elements are not asymptotically affine.  

\begin{figure}
\centering
\subfloat{
\begin{tikzpicture}
\begin{loglogaxis}[
	legend cell align=left,
	width=.49\textwidth,
    title={Convergence (structured meshes)},
    xlabel={Mesh size $h$},
    ylabel={$L^2$ error},
    xmin=.01, xmax=100,
    ymin=5e-7, ymax=2,        
    legend pos=south east,
    xmajorgrids=true,
    ymajorgrids=true,
    grid style=dashed,
] 
\addplot+[color=blue,mark=*,mark options={fill=markercolor},semithick]
coordinates{(2,0.969171)(1,0.46998)(0.5,0.14959)(0.25,0.0386321)(0.125,0.00957494)};
\addplot+[color=red,mark=square*,mark options={fill=markercolor},semithick]
coordinates{(2,0.229495)(1,0.0485385)(0.5,0.00502668)(0.25,0.000568599)(0.125,6.91e-05)};
\addplot+[color=black,mark=triangle*,mark options={fill=markercolor},semithick]
coordinates{(2,0.1133)(1,0.00628775)(0.5,0.000397568)(0.25,2.54e-05)(0.125,1.7e-06)};

\addplot+[color=blue,dashed,mark=*,mark options={solid,fill=markercolor},semithick]
coordinates{(2,0.813)(1,0.197)(0.5,0.0489)(0.25,0.0115)(0.125,0.00277)};
\addplot+[color=red,dashed,mark=square*,mark options={solid,fill=markercolor},semithick]
coordinates{(2,0.165)(1,0.035)(0.5,0.00442)(0.25,0.000549)(0.125,6.85e-05)};
\addplot+[color=black,dashed,mark=triangle*,mark options={solid,fill=markercolor},semithick]
coordinates{(2,0.108)(1,0.0061)(0.5,0.000366)(0.25,2.3e-05)(0.125,1.84e-06)};
\legend{$N=1 \text{ (DG)}$,$N=2 \text{ (DG)}$,$N=3 \text{ (DG)}$,$N=1 \text{ (LSC)}$,$N=2 \text{ (LSC)}$,$N=3 \text{ (LSC)}$}
\end{loglogaxis}
\end{tikzpicture}
}
\subfloat{
\begin{tikzpicture}
\begin{loglogaxis}[
	legend cell align=left,
	width=.49\textwidth,
    title={Convergence (unstructured meshes)},
    xlabel={Mesh size $h$},
    ylabel={$L^2$ error},
    xmin=.01, xmax=100,
    ymin=5e-7, ymax=2,    
    legend pos=south east,
    xmajorgrids=true,
    ymajorgrids=true,
    grid style=dashed,
] 
\addplot+[color=blue,mark=*,mark options={fill=markercolor},semithick]
coordinates{(2,0.969)(1,0.567)(0.5,0.115)(0.25,0.0356)(0.125,0.00864)};
\addplot+[color=red,mark=square*,mark options={fill=markercolor},semithick]
coordinates{(2,0.229)(1,0.0329)(0.5,0.00274)(0.25,0.000407)(0.125,4.92e-05)};
\addplot+[color=black,mark=triangle*,mark options={fill=markercolor},semithick]
coordinates{(2,0.113)(1,0.00834)(0.5,0.000169)(0.25,1.53e-05)(0.125,1.55e-06)};

\addplot+[dashed,color=blue,mark=*,mark options={solid,fill=markercolor},semithick]
coordinates{(2,0.813)(1,0.235)(0.5,0.0342)(0.25,0.0103)(0.125,0.00334)};
\addplot+[dashed,color=red,mark=square*,mark options={solid,fill=markercolor},semithick]
coordinates{(2,0.165)(1,0.0274)(0.5,0.00252)(0.25,0.00046)(0.125,0.000102)};
\addplot+[dashed,color=black,mark=triangle*,mark options={solid,fill=markercolor},semithick]
coordinates{(2,0.108)(1,0.0077)(0.5,0.000162)(0.25,1.68e-05)(0.125,7.35e-06)};

\legend{$N=1 \text{ (DG)}$,$N=2 \text{ (DG)}$,$N=3 \text{ (DG)}$,$N=1 \text{ (LSC)}$,$N=2 \text{ (LSC)}$,$N=3 \text{ (LSC)}$}
\end{loglogaxis}
\end{tikzpicture}
}
\caption{Convergence of the $L^2$ error for LSC-DG and DG under structured and unstructured wedge mesh refinement.}
\label{fig:rates}
\end{figure}
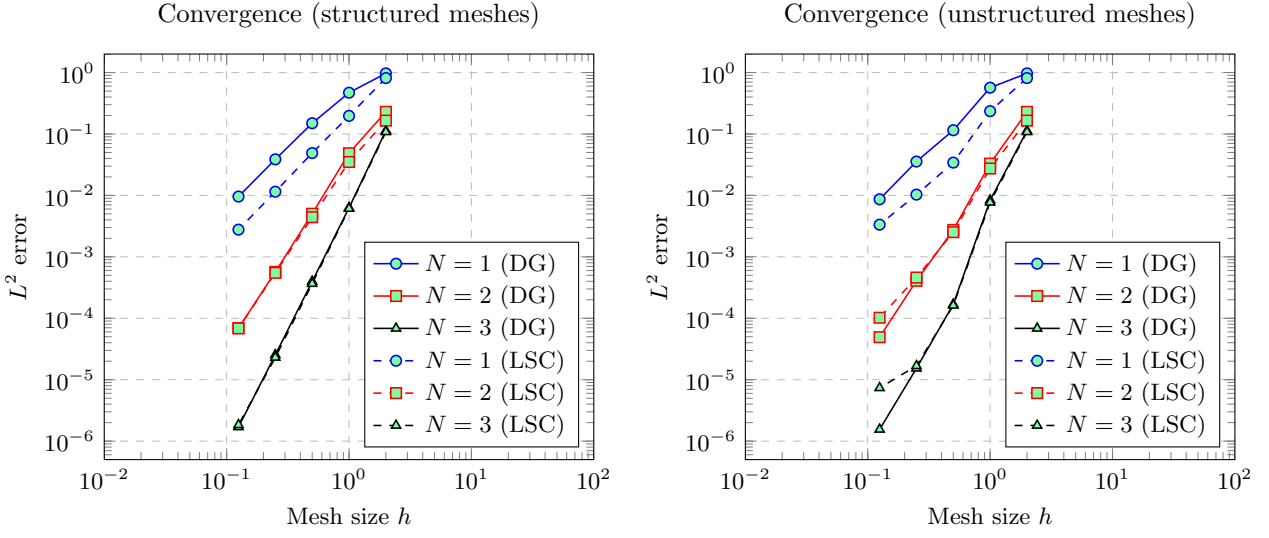

To emphasize that the convergence of LSC-DG can stagnate under specific types of mesh refinement, we consider a sequence of vertically mapped wedge meshes based on the quadrilateral meshes used in \cite{arnold2002approximation,warburton2013low}.  Each mesh in this sequence is constructed through a self-similar refinement pattern.  Figure~\ref{fig:arnold} shows an example of one such mesh, as well as the $L^2$ errors of DG and LSC-DG on such meshes.  While $L^2$ errors for DG are similar to those observed under structured mesh refinement, errors for LSC-DG stagnate under mesh refinement.  

\begin{figure}
\centering
\subfloat{
\begin{tikzpicture}
\begin{loglogaxis}[
	legend cell align=left,
	width=.49\textwidth,
    title={Convergence (structured meshes)},
    xlabel={Mesh size $h$},
    ylabel={$L^2$ error},
    xmin=.01, xmax=100,
    ymin=5e-7, ymax=2,        
    legend pos=south east,
    xmajorgrids=true,
    ymajorgrids=true,
    grid style=dashed,
] 
\addplot+[color=blue,mark=*,mark options={fill=markercolor},semithick]
coordinates{(2,0.969)(1,0.769)(0.5,0.226)(0.25,0.0602)(0.125,0.0149)};
\addplot+[color=red,mark=square*,mark options={fill=markercolor},semithick]
coordinates{(2,0.229)(1,0.0845)(0.5,0.0105)(0.25,0.00113)(0.125,0.000128)};
\addplot+[color=black,mark=triangle*,mark options={fill=markercolor},semithick]
coordinates{(2,0.113)(1,0.0149)(0.5,0.000894)(0.25,5.83e-05)(0.125,3.66e-06)};

\addplot+[dashed,color=blue,mark=*,mark options={solid,fill=markercolor},semithick]
coordinates{(2,0.805)(1,0.364)(0.5,0.112)(0.25,0.0522)(0.125,0.0437)};
\addplot+[dashed,color=red,mark=square*,mark options={solid,fill=markercolor},semithick]
coordinates{(2,0.165)(1,0.0711)(0.5,0.0145)(0.25,0.00517)(0.125,0.00469)};
\addplot+[dashed,color=black,mark=triangle*,mark options={solid,fill=markercolor},semithick]
coordinates{(2,0.0195)(1,0.00293)(0.5,0.00041)(0.25,0.000399)(0.125,0.000401)};

\legend{$N=1 \text{ (DG)}$,$N=2 \text{ (DG)}$,$N=3 \text{ (DG)}$,$N=1 \text{ (LSC)}$,$N=2 \text{ (LSC)}$,$N=3 \text{ (LSC)}$}
\end{loglogaxis}
\end{tikzpicture}
}
\hspace{1em}
\subfloat{\includegraphics[width=.45\textwidth]{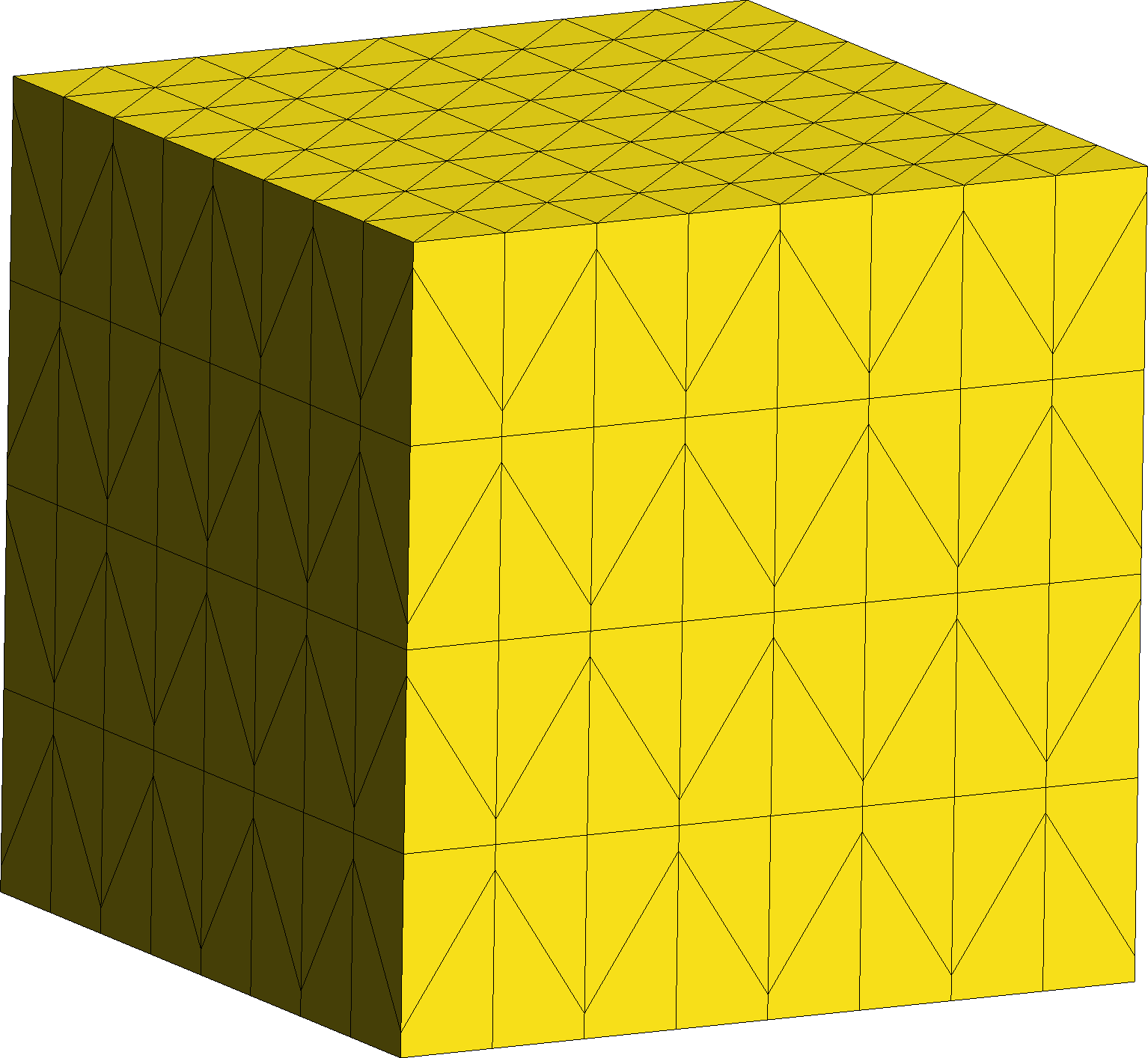}}
\caption{$L^2$ errors under mesh refinement for Arnold-type perturbed meshes.}
\label{fig:arnold}
\end{figure}
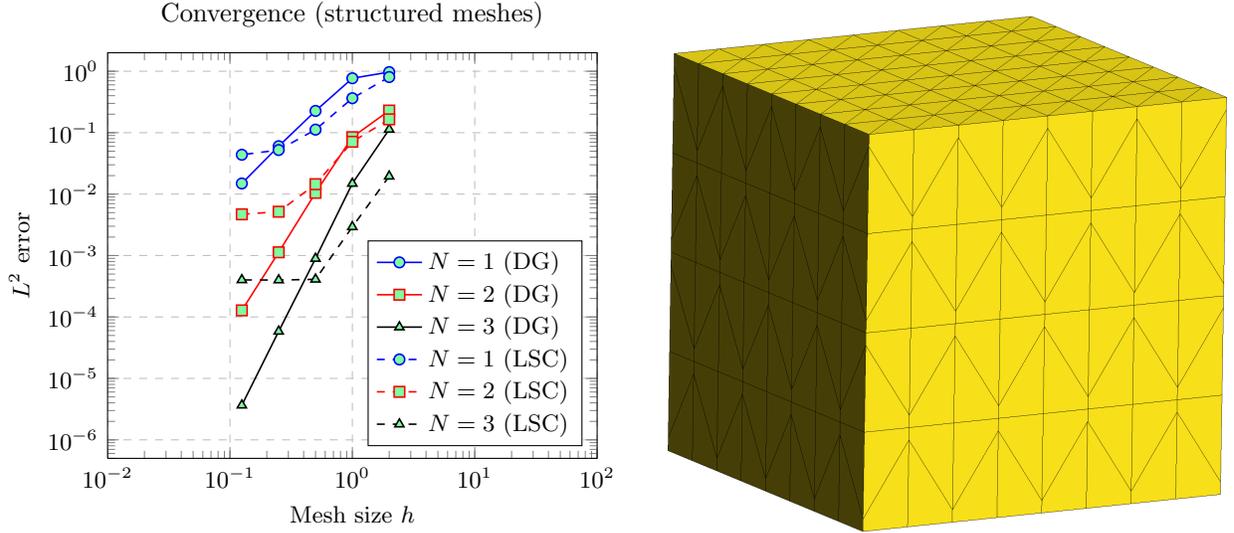

\begin{table}
\centering                                                                                   
\begin{tabular}{|c|c|c|c|}
\hline                                                   
 & $N=1$ & $N=2$ & $N=3$\\ 
\hline                                                   
Structured & 2.01 & 3.15 & 3.97\\ 
\hline                                                   
Unstructured & 1.72 & 2.9 & 4.42\\ 
\hline                                                   
Arnold-type & 1.9 & 3.13 & 3.99\\ 
\hline 
\end{tabular}
\caption{Observed rates of convergence of $L^2$ error for DG under various types of mesh refinement.}
\label{table:rates}
\end{table}

\subsection{Computational implementation}

Time-domain DG solvers using explicit time stepping typically advance the solution in time in two stages - a step which computes the evaluation of the right hand side, and one which computes the solution at the next timestep using a time integrator.  Following Kl\"ockner et al.\ \cite{klockner2009nodal}, we implement the above steps on a GPU by dividing this work into three distinct kernels for each type of element (wedge, tetrahedron):
\begin{itemize}
\item A volume kernel, which applies differentiation matrices to compute contributions to the right hand side resulting from volumetric integrals in the variational formulation.  
\item A surface kernel, which computes both numerical fluxes and applies the lift matrix to compute contributions to the right hand side resulting from surface integrals in the variational formulation.
\item An update kernel, which updates the solution in time.  
\end{itemize}
For the update kernel, we use a 3rd order multi-rate Adams Bashforth time integrator in this work \cite{godel2010gpu}, though any standard explicit time marching scheme may be used.  The implementation of wedge update kernels and tetrahedral volume, surface, and update kernels follows \cite{modave2015nodal}. Since the computational cost of the update kernel is minimal for quadrature-free nodal DG methods, we do not analyze its computational performance in this work.  Instead, we focus on the implementation of volume and surface kernels for the vertically mapped wedge elements.  

Finally, achieving good computational performance on GPUs typically requires tuning computational parameters to ensure work is distributed as evenly as possible among available GPU resources.  The work is divided such that a thread block processes computation for one or more elements.  We assign computational parameters $K_V, K_S$, and $K_U$, which we refer to respectively as the number of elements processed by the volume, surface, and update kernels in each thread block. These parameters can be tuned such that the number of total threads processed is made close to a multiple of 32, the number of threads which run concurrently within one thread block. For all following computational results, $K_V$, $K_S$, and $K_U$ have been optimized to minimize runtime.  

\subsection{GPU implementation}

In this section, we discuss parallelization strategies used in the implementation of volume and surface kernels on GPUs.  For the volume surface kernel, we parallelize over nodes in a triangular slice of the wedge, marching through slices of the wedge in serial.  This strategy takes advantage of the triangle-line tensor product structure of local matrices by decomposing computation into multiplication by one-dimensional and triangular matrices.  One-dimensional matrices are applied for each thread (corresponding to a node in a given triangular slice) by loading entries of the matrix from memory and computing the matrix product with data from the line of nodes in the extruded direction.  Since one-dimensional matrices are small, they are loaded into shared memory and reused for each triangular slice.  Triangular matrices are then applied by loading matrix entries from global memory once and reusing them for each slice of the wedge.   Algorithm~\ref{alg:vol} describes the implementation of the volume kernel in more detail.  

\begin{algorithm}%[H]
\caption{Wedge volume kernel} 
\label{alg:vol} 
\begin{algorithmic}[1]
\ParFor{each element $D^k$} 
\ParFor{triangular nodes $i = 1,\ldots, \frac{(N+1)(N+2)}{2}$} 
\For{one dimensional nodes $j = 1,\ldots, N+1$} 
\State Load geometric factors and solution variables from global memory into shared memory.
\State Load one-dimensional matrix $\bm{D}_t^{\rm 1D}$ to shared memory.
\EndFor
\EndParFor
\EndParFor

\State Memory fence (synchronize threads).  

\ParFor{each element $D^k$} 
\ParFor{triangular nodes $i = 1,\ldots, \frac{(N+1)(N+2)}{2}$} 
\State Apply one-dimensional matrix $\bm{D}_t^{\rm 1D}$ to the line of nodes extruded from node $i$.  
\State Store result to shared memory. % (requires two additional arrays of size $\frac{(N+1)^2(N+2)}{2}$).  
\EndParFor
\EndParFor

\State Memory fence (synchronize threads).  

\ParFor{each element $D^k$} 
\ParFor{triangular nodes $i = 1,\ldots, \frac{(N+1)(N+2)}{2}$} 
\State Load triangular matrices $\edit{\bm{L}^{{\rm tri},k}}, \bm{D}_r^{\rm tri},\bm{D}_s^{\rm tri}$ from global memory.
\For{one dimensional nodes $j = 1,\ldots, N+1$} 
\State Apply triangular matrices to the $j$th triangular slice.  
\State Accumulate results into thread-local memory.
\EndFor
\For{one dimensional nodes $j = 1,\ldots, N+1$} 
\State Load results from thread-local memory.
\State Assemble right hand side contributions and write to global memory.
\EndFor

\EndParFor
\EndParFor
\end{algorithmic}
\end{algorithm}

For the surface kernels, each thread in a thread block is assigned to compute the right hand side contribution for a single node in the wedge.  Since the lift matrices are Kronecker products of $\edit{\bm{L}^{{\rm tri},k}}$ and one-dimensional matrices or vectors, the computation is again broken up into an application of $\edit{\bm{L}^{{\rm tri},k}}$ and computations with one-dimensional matrices.  First, numerical fluxes are computed and premultiplied by the triangular lift $\edit{\bm{L}^{{\rm tri},k}}$, which is loaded from global memory.  These results are saved to shared memory, and used to compute the second half of the Kronecker product $\edit{\bm{L}^{{\rm tri},k}}\otimes \LRp{\edit{\widehat{\bm{M}}^{\rm 1D}}}^{-1}\bm{e}$, which corresponds to scaling of the numerical fluxes (after premultiplication by the triangular lift matrix) by the $j$th entry of the vector $\LRp{\edit{\widehat{\bm{M}}^{\rm 1D}}}^{-1}\bm{e}$ where $j$ corresponds to the index of the current triangular slice of the wedge.  The lift matrices for quadrilateral faces of the wedge are then directly applied to (non-premultiplied) numerical fluxes to complete the surface kernel.  This procedure is outlined in Algorithm~\ref{alg:surf}.

Since the loop sizes over which the surface kernel is parallelized vary between the number of total face nodes, the number of triangular nodes, and the number of nodes in the wedge, we take the number of threads to be the maximum of all these values.  Consequentially, some threads remain idle at each iteration.  However, experiments with parallelization over a smaller number of threads to reduce the number of idle threads (for example, over the number of face nodes in a triangular or quadrilateral face) resulted in slower overall runtimes.  

\begin{algorithm}%[H]
\caption{Wedge surface kernel} 
\label{alg:surf}
\begin{algorithmic}[1]
\ParFor{each element $D^k$} 
\ParFor{nodes $i = 1,\ldots, \text{total number of face nodes}$} 
\State Load geometric factors from global memory into shared memory.
\State Load quadrilateral face lift matrices into shared memory.
\EndParFor
\EndParFor

\State Memory fence (synchronize threads).  

\ParFor{each element $D^k$} 
\ParFor{nodes $i = 1,\ldots, \text{total number of face nodes}$} 
\State Load solution traces and compute numerical fluxes.
\EndParFor
\EndParFor

\State Memory fence (synchronize threads).  

\ParFor{each element $D^k$} 
\ParFor{triangular nodes $i = 1,\ldots, \frac{(N+1)(N+2)}{2}$} 
\State Premultiply numerical fluxes on triangular faces by $J^k_f \edit{\bm{L}^{{\rm tri},k}}$.
\State Store result to shared memory.  
\EndParFor
\EndParFor

\State Memory fence (synchronize threads).  

\ParFor{each element $D^k$} 
\ParFor{wedge nodes $i = 1,\ldots, N_p$} 
\State Let $j$ to be the index of the triangular slice node $i$ lies in.  
\State Load premultiplied numerical fluxes on triangular faces, scale by $j$th entry of vector $\LRp{\edit{\widehat{\bm{M}}^{\rm 1D}}}^{-1} \bm{e}$.   
\State Apply lift matrices to numerical fluxes on quadrilateral faces.  
\State Assemble right hand side contributions and write to global memory.
\EndParFor
\EndParFor
\end{algorithmic}
\end{algorithm}

\subsection{Computational evaluation and performance}

%\edit{add LSC-DG comparison to emphasize that these wedges are cheaper.}

In this section, we present results and timings which quantify the efficiency of a GPU-accelerated implementation of a nodal discontinuous Galerkin solver for vertically mapped wedge meshes.  All computational results are reported for single precision.  \edit{For the simulation of wave propagation in seismic applications, single precision is sufficient for both accuracy and numerical stability.  For more complex nonlinear physics, some settings with highly disparate length scales benefit more from double precision.  The use of double precision tends to result in slightly less than half the performance of GPU kernels in single precision, though the precise performance difference depends on the architecture of the specific GPU. }

The computational complexity of both the volume and surface kernel for a given wedge element is $O(N^4)$.  For the volume kernel, the cost is dominated by the application of the differentiation matrices.  Since each differentiation matrix is a Kronecker product of a $\frac{(N+1)(N+2)}{2} \times \frac{(N+1)(N+2)}{2}$ triangular matrix and a one-dimensional $(N+1)\times (N+1)$ matrix, the cost is dominated by the $O(N^4)$ cost of applying triangular matrices.  Similarly, for the surface kernel, the cost is dominated by the application of the lift matrices for each triangular face.  These matrices are again representable as Kronecker products of triangular and one-dimensional matrices, implying that the asymptotic cost of the surface kernel is also $O(N^4)$.  

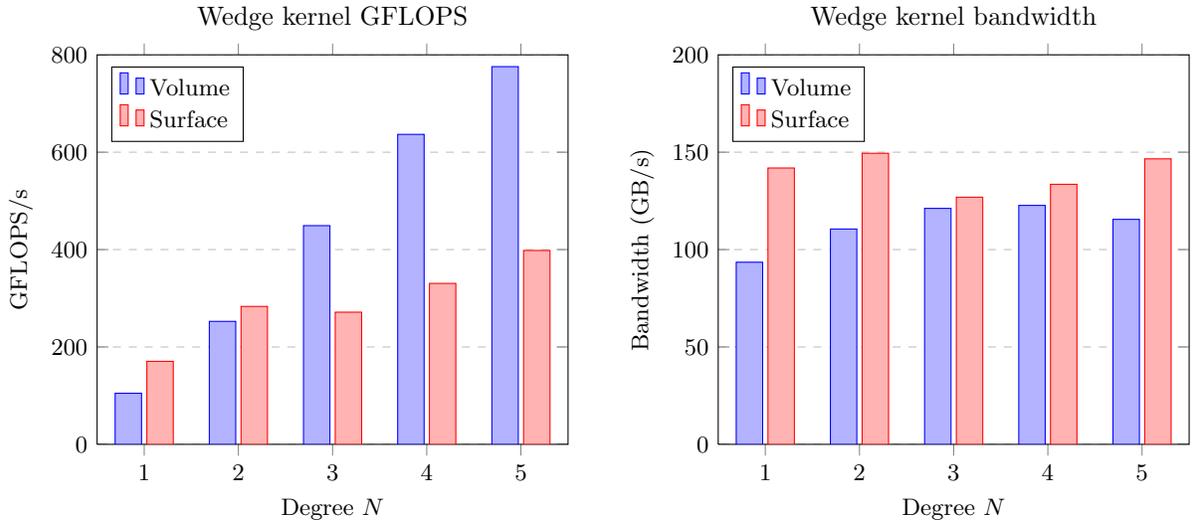
\begin{figure}
\centering
\subfloat{
\begin{tikzpicture}
\begin{axis}[
	width=.475\textwidth,
	legend cell align=left,
	title={Wedge kernel GFLOPS },
	xlabel={Degree $N$},
	ylabel={GFLOPS/s},
	xmin=.5, xmax=5.5,
	ymin=0,ymax=800,
	%ybar stacked,
        ybar=5*\pgflinewidth,
        bar width=10pt,	
	xtick={1,2,3,4,5},
	legend pos=north west,
%	xmajorgrids=true,
	ymajorgrids=true,
	grid style=dashed,
] 
%nodal runtimes
\addplot coordinates{(1,104.907)(2,252.584)(3,449.157)(4,636.643)(5,775.717)};
\addplot coordinates{(1,170.494)(2,283.227)(3,271.461)(4,330.498)(5,398.135)};

\legend{Volume, Surface}
\end{axis}
\end{tikzpicture}
}
\hspace{1em}
\subfloat{
\begin{tikzpicture}
\begin{axis}[
	width=.475\textwidth,
	legend cell align=left,
	title={Wedge kernel bandwidth},
	xlabel={Degree $N$},
	ylabel={Bandwidth (GB/s)},
	xmin=.5, xmax=5.5,
	ymin=0,ymax=200,	
        ybar=5*\pgflinewidth,
        bar width=10pt,
	xtick={1,2,3,4,5},
	ymin=0,
	legend pos=north west,
	ymajorgrids=true,
	grid style=dashed,
] 
%speedup V/S
\addplot coordinates{(1,93.507)(2,110.555)(3,121.154)(4,122.683)(5,115.509)};
\addplot coordinates{(1,141.882)(2,149.438)(3,126.89)(4,133.512)(5,146.584)};

\legend{Volume, Surface}
\end{axis}
\end{tikzpicture}
}
\caption{Computational performance of wedge kernels.}
\end{figure}

For both the volume and surface kernels, computational performance comparable to standard volume and surface kernels for nodal discontinuous Galerkin methods on tetrahedra is observed \cite{klockner2009nodal, chan2015bbdg}, though the achieved GFLOPS are slightly lower.  In contrast, the observed bandwidth of the wedge volume and surface kernels is higher than that observed for nodal DG methods on tetrahedra by a factor of two for the volume kernel and a factor of slightly less than three for the surface kernel at $N=5$.  
%This difference is due to the fact that for nodal DG on tetrahedra, memory requirements for elemental matrices grow as $O(N^{6})$ \cite{chan2015bbdg}.  For wedges, this bottleneck is alleviated since only triangular and one-dimensional matrices grow more slowly with $N$ ($O(N^4)$ and $O(N^2)$, respectively).  
%This difference is due to the $O(N^6)$ computational complexity for nodal DG on tetrahedra  \cite{chan2015bbdg}, which 
%For wedges, this bottleneck is alleviated since the cost of multiplication by triangular and one-dimensional matrices grows more slowly with $N$ ($O(N^4)$ and $O(N^2)$, respectively).  

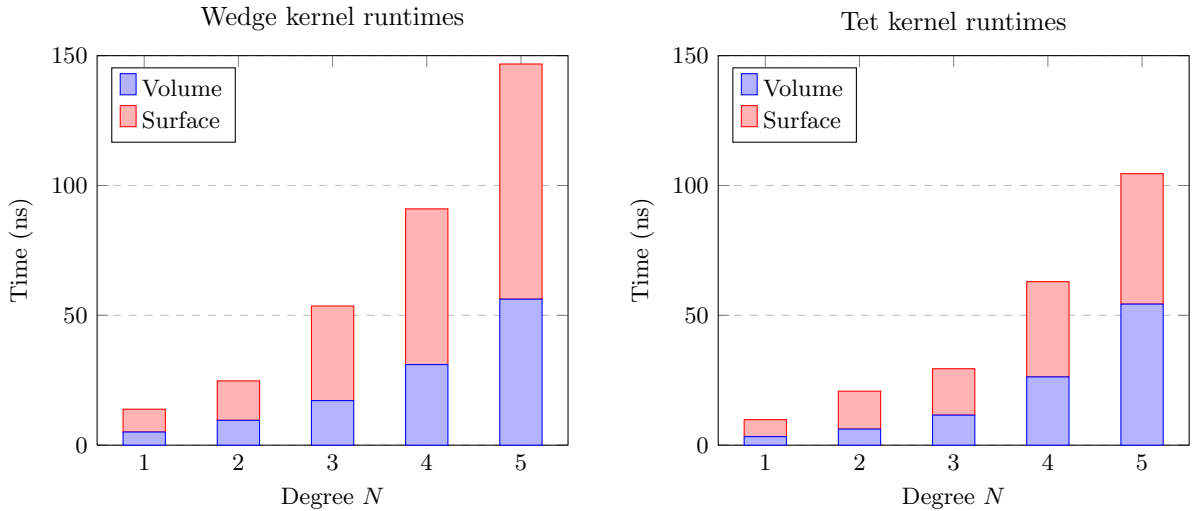
\begin{figure}
\centering
\subfloat{
\begin{tikzpicture}
\begin{axis}[
	width=.475\textwidth,
	legend cell align=left,
	title={Wedge kernel runtimes},
	xlabel={Degree $N$},
	ylabel={Time (ns)},
	xmin=.5, xmax=5.5,
%	ymin= 0,ymax=1.5e-7,
	ymin= 0,ymax=150,
	ybar stacked,
%       ybar=5*\pgflinewidth,
        bar width=16pt,	
	xtick={1,2,3,4,5},
	legend pos=north west,
	ymajorgrids=true,
	grid style=dashed,
] 
%nodal runtimes
%\addplot coordinates{(1,5.06162e-09)(2,9.57306e-09)(3,1.71655e-08)(4,3.10299e-08)(5,5.6228e-08)};
\addplot coordinates{(1,5.06162)(2,9.57306)(3,17.1655)(4,31.0299)(5,56.228)};
%\addplot coordinates{(1,8.80282e-09)(2,1.51849e-08)(3,3.64217e-08)(4,5.99692e-08)(5,9.0559e-08)};
\addplot coordinates{(1,8.80282)(2,15.1849)(3,36.4217)(4,59.9692)(5,90.559)};
\legend{Volume, Surface}
\end{axis}
\end{tikzpicture}
}
\hspace{1em}
\subfloat{
\begin{tikzpicture}
\begin{axis}[
	width=.475\textwidth,
	legend cell align=left,
	title={Tet kernel runtimes},
	xlabel={Degree $N$},
	ylabel={Time (ns)},
	xmin=.5, xmax=5.5,
%	ymin=0,ymax=1.5e-7,	
	ymin=0,ymax=150,	
	ybar stacked,	
%        ybar=5*\pgflinewidth,
        bar width=16pt,
	xtick={1,2,3,4,5},
	ymin=0,
	legend pos=north west,
	ymajorgrids=true,
	grid style=dashed,
] 
%speedup V/S
%\addplot coordinates{(1,3.27103e-09)(2,6.19159e-09)(3,1.15654e-08)(4,2.6285e-08)(5,5.43224e-08)};
%\addplot coordinates{(1,6.54206e-09)(2,1.46028e-08)(3,1.78738e-08)(4,3.66822e-08)(5,5.02336e-08)};
\addplot coordinates{(1,3.27103)(2,6.19159)(3,11.5654)(4,26.285)(5,54.3224)};
\addplot coordinates{(1,6.54206)(2,14.6028)(3,17.8738)(4,36.6822)(5,50.2336)};

\legend{Volume, Surface}
\end{axis}
\end{tikzpicture}
}
\caption{Per-element runtimes for wedge and tet kernels.} 
\label{fig:runtime}
\end{figure}

Finally, comparing wall-clock runtimes of both wedge and tetrahedral kernels in Figure~\ref{fig:runtime}, we find that the per-element runtime of the wedge kernels is always greater than or equal to that of the tetrahedral kernels.  However, this can be misleading, as the number of degrees of freedom for a high order wedge is greater than that of a high order tetrahedra.  This is also reflected in the fact that a given volume can typically be triangulated with fewer wedges than tetrahedra; for example, a cube may be split into two well-conditioned prisms or five well-conditioned tetrahedra.  

Normalizing by the number of degrees of freedom in Figure~\ref{fig:rundof} shows that the per-dof runtime of the wedge volume kernel is lower than the per-dof runtime of the tetrahedral volume kernel for $N>1$.  For the surface kernel, the per-dof runtime is less for all orders except $N = 3$, for which tetrahedral kernels have been observed to perform better due to the number of degrees of freedom being close to a multiple of the number of concurrently executable threads in a workgroup \cite{klockner2009nodal}.  

We observe that the wedge volume kernel becomes more and more efficient relative to the tetrahedral volume kernel as $N$ increases; this is due to the fact that the asymptotic complexity volume kernel is $O(N^4)$ for the wedge, as opposed to $O(N^6)$ for tetrahedra.  We do not expect the same speedup as $N$ increases for the surface kernel, since both the wedge and tetrahedral surface kernels have the same asymptotic cost of $O(N^4)$.  

%\edit{On the hybrid mesh we generated, over twice as many tetrahedra as prisms to mesh a roughly equivalent volume.}

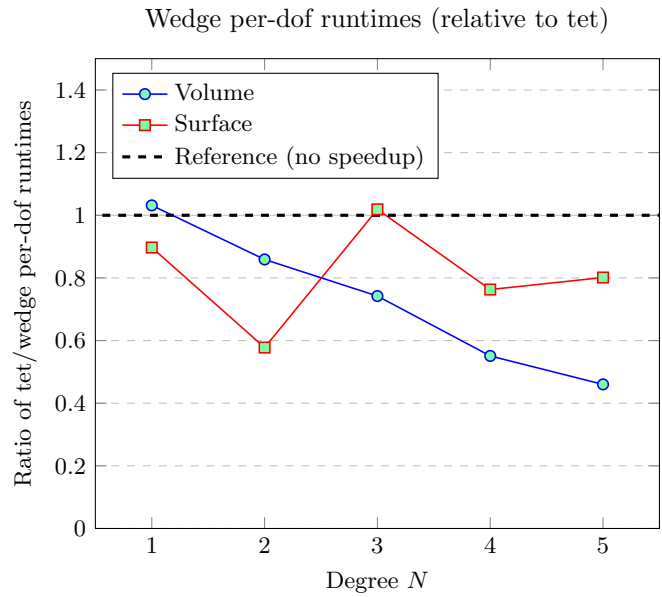
\begin{figure}
\centering
\begin{tikzpicture}
\begin{axis}[
	width=.55\textwidth,
	legend cell align=left,
	title={Wedge per-dof runtimes (relative to tet)},
	xlabel={Degree $N$},
	ylabel={Ratio of tet/wedge per-dof runtimes},
	xmin=.5, xmax=5.5,
	ymin=0,ymax=1.5,
%       ybar=5*\pgflinewidth,
%        bar width=10pt,	
	xtick={1,2,3,4,5},
	legend pos=north west,
	ymajorgrids=true,
	grid style=dashed,
] 
%nodal runtimes
\addplot+[color=blue,mark=*,mark options={fill=markercolor},semithick]
coordinates{(1,1.03161)(2,0.858967)(3,0.742104)(4,0.550908)(5,0.460035)};
\addplot+[color=red,mark=square*,mark options={fill=markercolor},semithick]
coordinates{(1,0.897049)(2,0.5777)(3,1.01885)(4,0.76292)(5,0.801225)};
\addplot+[draw=black,mark=none,line legend, very thick,smooth,dashed] coordinates{(0,1)(6,1)};
\legend{Volume, Surface, Reference (no speedup)}
\end{axis}
\end{tikzpicture}
\caption{Ratio of per-dof runtimes of wedge kernels relative to per-dof runtimes of tetrahedral kernels (lower is better).}
\label{fig:rundof}
\end{figure}

We conclude by presenting a computational simulation of wave propagation through the bi-unit cube $[-1,1]^3$ with a wavespeed which is discontinuous across a complex interface.  The domain is triangulated using a hybrid mesh consisting of vertically mapped wedges and unstructured tetrahedra, and the initial condition is set to be a Gaussian pulse located at the origin $(0,0,0)$.  Figure~\ref{fig:wavy} shows the pressure at time $t = .75$ computed using a DG method with order $N=4$.  

\begin{figure}
\centering
%\subfloat{\includegraphics[height=.3\textwidth]{p12.png}}
%\subfloat{\includegraphics[height=.3\textwidth]{pScaled12.png}}
\includegraphics[width=.8\textwidth]{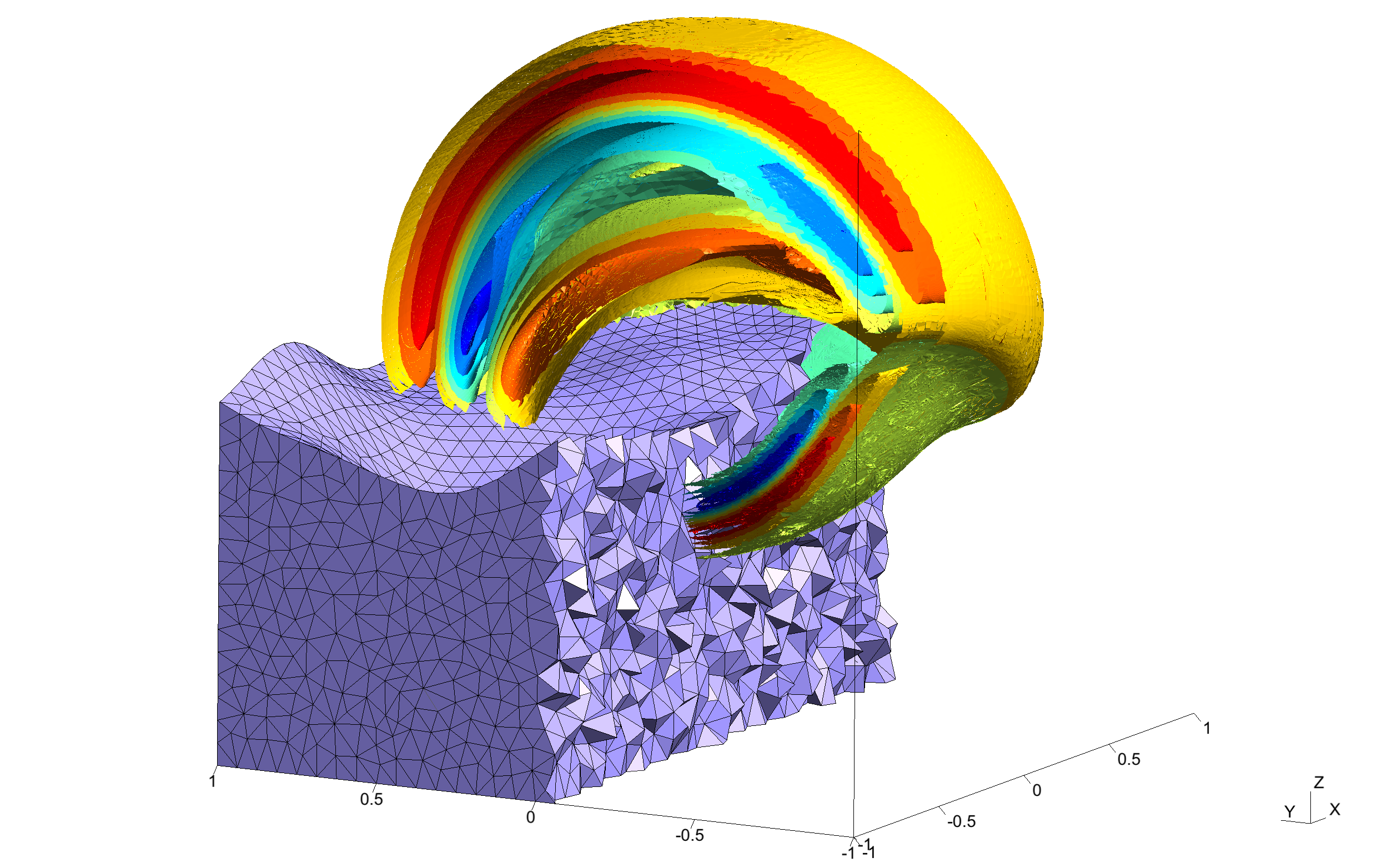}
\caption{Numerical solution of the wave equation for a hybrid wedge-tet mesh ($N=4$) with discontinuous wavespeed across a wavy interface. }
\label{fig:wavy}
\end{figure}

\section{Conclusions and future work}

Existing options for efficient low-storage discontinuous Galerkin methods on meshes with general mapped wedges are currently limited to Low-storage Curvilinear DG methods.  Since LSC-DG methods are typically more expensive and less robust to mesh perturbations than quadrature-free nodal DG methods, we have developed in this work a limited-storage nodal DG method for vertically mapped wedges to address these issues.  This method is also energy stable, in contrast to the mass-lumped methods typically employed for high order finite element methods on wedges.  Additional interpolation to quadrature points is also avoided, despite the fact that exact quadrature is used.  Finally, the method is shown to be computationally efficient, with kernels attaining good computational performance and bandwidth at high orders of approximation, while reducing the per-dof runtime relative to tetrahedral kernels.  

%It is possible to achieve further speedups by employing a Bernstein-Bezier basis on the wedge as opposed to a nodal basis \cite{chan2015bbdg}, which would reduce the computational complexity of the applying differentiation matrices from $O(N^3)$ from $O(N^4)$ by exposing sparsity via a change of basis.  
It is possible to achieve further speedup by taking advantage of the fact that the wedge meshes used in this work are structured in the vertical direction.  This can be exploited to increase data locality by utilizing a layer-by-layer marching algorithm and reusing data from the previous layer \cite{micikevicius20093d,bercea2015numbering}.  

%Curvilinear meshes and interface errors \cite{wang2009discontinuous}.

\section{Acknowledgements}

The authors thank TOTAL for permission to publish.  JC and TW are funded by a grant from TOTAL E\&P Research and Technology USA.  Solutions were rendered using the high order visualization of GMSH \cite{geuzaine2009gmsh}.

\bibliographystyle{unsrt}
\bibliography{total}{}

\begin{thebibliography}{10}

\bibitem{komatitsch1999introduction}
Dimitri Komatitsch and Jeroen Tromp.
\newblock Introduction to the spectral element method for three-dimensional
  seismic wave propagation.
\newblock {\em Geophysical journal international}, 139(3):806--822, 1999.

\bibitem{karniadakis2013spectral}
George Karniadakis and Spencer Sherwin.
\newblock {\em Spectral/hp element methods for computational fluid dynamics}.
\newblock Oxford University Press, 2013.

\bibitem{hesthaven2007nodal}
Jan~S Hesthaven and Tim Warburton.
\newblock {\em Nodal discontinuous Galerkin methods: algorithms, analysis, and
  applications}, volume~54.
\newblock Springer, 2007.

\bibitem{klockner2009nodal}
Andreas Kl{\"o}ckner, Tim Warburton, Jeff Bridge, and Jan~S Hesthaven.
\newblock Nodal discontinuous {Galerkin} methods on graphics processors.
\newblock {\em Journal of Computational Physics}, 228(21):7863--7882, 2009.

\bibitem{godel2010scalability}
Nico G{\"o}del, Nigel Nunn, Tim Warburton, and Markus Clemens.
\newblock Scalability of higher-order discontinuous {G}alerkin {FEM}
  computations for solving electromagnetic wave propagation problems on {GPU}
  clusters.
\newblock {\em Magnetics, IEEE Transactions on}, 46(8):3469--3472, 2010.

\bibitem{modave2015nodal}
A~Modave, A~St-Cyr, WA~Mulder, and T~Warburton.
\newblock A nodal discontinuous galerkin method for reverse-time migration on
  gpu clusters.
\newblock {\em Geophysical Journal International}, 203(2):1419--1435, 2015.

\bibitem{modave2016gpu}
Axel Modave, Amik St-Cyr, and Tim Warburton.
\newblock Gpu performance analysis of a nodal discontinuous galerkin method for
  acoustic and elastic models.
\newblock {\em Computers \& Geosciences}, 91:64--76, 2016.

\bibitem{chan2015gpu}
Jesse Chan, Zheng Wang, Axel Modave, Jean-Francois Remacle, and T~Warburton.
\newblock {GPU-accelerated discontinuous Galerkin methods on hybrid meshes}.
\newblock {\em Journal of Computational Physics}, 318:142--168, 2016.

\bibitem{warburton2013low}
T.~Warburton.
\newblock A low-storage curvilinear discontinuous {G}alerkin method for wave
  problems.
\newblock {\em SIAM Journal on Scientific Computing}, 35(4):A1987--A2012, 2013.

\bibitem{lapin2003joint}
Serguei Lapin, Sonja Kisin, Hua-wei Zhou, et~al.
\newblock Joint {VSP} and surface seismic tomography.
\newblock In {\em 2003 SEG Annual Meeting}. Society of Exploration
  Geophysicists, 2003.

\bibitem{wang2008finite}
Qiang Wang, Sergey Danilov, and Jens Schr{\"o}ter.
\newblock Finite element ocean circulation model based on triangular prismatic
  elements, with application in studying the effect of topography
  representation.
\newblock {\em Journal of Geophysical Research: Oceans}, 113(C5), 2008.

\bibitem{blaise2010discontinuous}
S{\'e}bastien Blaise, Richard Comblen, Vincent Legat, Jean-Fran{\c{c}}ois
  Remacle, Eric Deleersnijder, and Jonathan Lambrechts.
\newblock A discontinuous finite element baroclinic marine model on
  unstructured prismatic meshes.
\newblock {\em Ocean Dynamics}, 60(6):1371--1393, 2010.

\bibitem{botti2012influence}
Lorenzo Botti.
\newblock Influence of reference-to-physical frame mappings on approximation
  properties of discontinuous piecewise polynomial spaces.
\newblock {\em Journal of Scientific Computing}, 52(3):675--703, 2012.

\bibitem{warburton2006explicit}
T~Warburton.
\newblock An explicit construction of interpolation nodes on the simplex.
\newblock {\em Journal of engineering mathematics}, 56(3):247--262, 2006.

\bibitem{kirby2000discontinuous}
RM~Kirby, TC~Warburton, I~Lomtev, and GE~Karniadakis.
\newblock A discontinuous {Galerkin} spectral/hp method on hybrid grids.
\newblock {\em Applied numerical mathematics}, 33(1):393--405, 2000.

\bibitem{bergot2013higher}
Morgane Bergot and Marc Durufl{\'e}.
\newblock Higher-order discontinuous {Galerkin} method for pyramidal elements
  using orthogonal bases.
\newblock {\em Numerical Methods for Partial Differential Equations},
  29(1):144--169, 2013.

\bibitem{kopriva2010quadrature}
David~A Kopriva and Gregor Gassner.
\newblock On the quadrature and weak form choices in collocation type
  discontinuous {Galerkin} spectral element methods.
\newblock {\em Journal of Scientific Computing}, 44(2):136--155, 2010.

\bibitem{chan2016weight1}
Jesse Chan, Russell~J Hewett, and T~Warburton.
\newblock {Weight-adjusted discontinuous Galerkin methods: wave propagation in
  heterogeneous media}.
\newblock {\em arXiv preprint arXiv:1608.01944}, 2016.

\bibitem{hesthaven2002nodal}
Jan~S Hesthaven and T.~Warburton.
\newblock Nodal high-order methods on unstructured grids: {I}. time-domain
  solution of {M}axwell's equations.
\newblock {\em Journal of Computational Physics}, 181(1):186--221, 2002.

\bibitem{hesthaven2004high}
Jan~S Hesthaven and T.~Warburton.
\newblock High-order nodal discontinuous {G}alerkin methods for the {M}axwell
  eigenvalue problem.
\newblock {\em Philosophical Transactions of the Royal Society of London A:
  Mathematical, Physical and Engineering Sciences}, 362(1816):493--524, 2004.

\bibitem{gassner2013skew}
Gregor~J Gassner.
\newblock A skew-symmetric discontinuous {Galerkin} spectral element
  discretization and its relation to {SBP-SAT} finite difference methods.
\newblock {\em SIAM Journal on Scientific Computing}, 35(3):A1233--A1253, 2013.

\bibitem{gandham2015high}
Rajesh Gandham.
\newblock {\em High performance high-order numerical methods: applications in
  ocean modeling}.
\newblock PhD thesis, Rice University, 2015.

\bibitem{brenner2007mathematical}
Susanne Brenner and Ridgway Scott.
\newblock {\em The mathematical theory of finite element methods}, volume~15.
\newblock Springer Science \& Business Media, 2007.

\bibitem{chan2015orthogonal}
Jesse Chan and T~Warburton.
\newblock Orthogonal bases for vertex-mapped pyramids.
\newblock {\em SIAM Journal on Scientific Computing}, 38(2):A1146--A1170, 2016.

\bibitem{wan2013icon}
Hui Wan, Marco~A Giorgetta, G{\"u}nther Z{\"a}ngl, Marco Restelli, Detlev
  Majewski, Luca Bonaventura, Kristina Fr{\"o}hlich, Daniel Reinert,
  P~R{\i}podas, Luis Kornblueh, et~al.
\newblock {The ICON-1.2 hydrostatic atmospheric dynamical core on triangular
  grids, Part I: formulation and performance of the baseline version}.
\newblock {\em Geoscientific Model Development}, 6:735--763, 2013.

\bibitem{casulli2000unstructured}
Vincenzo Casulli and Roy~A. Walters.
\newblock An unstructured grid, three-dimensional model based on the shallow
  water equations.
\newblock {\em International journal for numerical methods in fluids},
  32(3):331--348, 2000.

\bibitem{phillips1957coordinate}
Norman~A Phillips.
\newblock A coordinate system having some special advantages for numerical
  forecasting.
\newblock {\em Journal of Meteorology}, 14(2):184--185, 1957.

\bibitem{si2015tetgen}
Hang Si.
\newblock {TetGen, a Delaunay-based quality tetrahedral mesh generator}.
\newblock {\em ACM Transactions on Mathematical Software (TOMS)}, 41(2):11,
  2015.

\bibitem{warburton2010low}
T~Warburton.
\newblock A low storage curvilinear discontinuous {G}alerkin time-domain method
  for electromagnetics.
\newblock In {\em Electromagnetic Theory (EMTS), 2010 URSI International
  Symposium on}, pages 996--999. IEEE, 2010.

\bibitem{medina2014occa}
David~S Medina, Amik St-Cyr, and T~Warburton.
\newblock {OCCA}: A unified approach to multi-threading languages.
\newblock {\em arXiv preprint arXiv:1403.0968}, 2014.

\bibitem{cockburn2008optimal}
Bernardo Cockburn, Bo~Dong, and Johnny Guzm{\'a}n.
\newblock Optimal convergence of the original {DG} method for the
  transport-reaction equation on special meshes.
\newblock {\em SIAM Journal on Numerical Analysis}, 46(3):1250--1265, 2008.

\bibitem{johnson1986analysis}
Claes Johnson and Juhani Pitk{\"a}ranta.
\newblock An analysis of the discontinuous {Galerkin} method for a scalar
  hyperbolic equation.
\newblock {\em Mathematics of computation}, 46(173):1--26, 1986.

\bibitem{wilcox2010high}
Lucas~C Wilcox, Georg Stadler, Carsten Burstedde, and Omar Ghattas.
\newblock A high-order discontinuous {Galerkin} method for wave propagation
  through coupled elastic--acoustic media.
\newblock {\em Journal of Computational Physics}, 229(24):9373--9396, 2010.

\bibitem{arnold2002approximation}
Douglas Arnold, Daniele Boffi, and Richard Falk.
\newblock Approximation by quadrilateral finite elements.
\newblock {\em Mathematics of computation}, 71(239):909--922, 2002.

\bibitem{godel2010gpu}
Nico G{\"o}del, Steffen Schomann, Tim Warburton, and Markus Clemens.
\newblock {GPU} accelerated {A}dams--{B}ashforth multirate discontinuous
  {Galerkin FEM} simulation of high-frequency electromagnetic fields.
\newblock {\em Magnetics, IEEE Transactions on}, 46(8):2735--2738, 2010.

\bibitem{chan2015bbdg}
Jesse Chan and T~Warburton.
\newblock {GPU}-accelerated {Bernstein-Bezier} discontinuous {Galerkin} methods
  for wave problems.
\newblock {\em arXiv preprint arXiv:1512.06025}, 2015.

\bibitem{micikevicius20093d}
Paulius Micikevicius.
\newblock 3{D} finite difference computation on {GPUs} using {CUDA}.
\newblock In {\em Proceedings of 2nd workshop on general purpose processing on
  graphics processing units}, pages 79--84. ACM, 2009.

\bibitem{bercea2015numbering}
Gheorghe-Teodor Bercea, Andrew T.~T. McRae, David~A. Ham, Lawrence Mitchell,
  Florian Rathgeber, Luigi Nardi, Fabio Luporini, and Paul H.~J. Kelly.
\newblock A numbering algorithm for finite element on extruded meshes which
  avoids the unstructured mesh penalty.
\newblock {\em arXiv preprint arXiv:1604.05937}, 2015.

\bibitem{geuzaine2009gmsh}
Christophe Geuzaine and Jean-Fran{\c{c}}ois Remacle.
\newblock {Gmsh: A 3-D} finite element mesh generator with built-in pre-and
  post-processing facilities.
\newblock {\em International Journal for Numerical Methods in Engineering},
  79(11):1309--1331, 2009.

\end{thebibliography}

\end{document}